\documentclass[a4paper,11pt]{article}
\usepackage{a4wide}
\usepackage{cite}
\usepackage[utf8]{inputenc}
\usepackage[T1]{fontenc}

\usepackage{amsmath}
\usepackage{amscd}
\usepackage{amssymb}
\usepackage{amsrefs}
\usepackage{amsthm}
\usepackage{cases}
\usepackage{framed}
\usepackage{graphicx}
\usepackage{latexsym}
\usepackage{color}
\usepackage{array,multirow,makecell}

\setcellgapes{1pt}
\makegapedcells

\theoremstyle{plain} 
\newtheorem{theorem}{Theorem}[section]
\newtheorem{proposition}[theorem]{Proposition}
\newtheorem{lemma}[theorem]{Lemma} 

\theoremstyle{definition} 
\newtheorem{definition}[theorem]{Definition}

\theoremstyle{remark} 
\newtheorem{remark}[theorem]{Remark}

\renewcommand{\P}{\mathrm{P}}
\newcommand{\la}{\langle \,}
\newcommand{\ra}{\,\rangle }

\newcommand{\R}{\mathbb{R}}
\newcommand{\C}{\mathcal{C}}
\renewcommand{\L}{\mathcal{L}}
\renewcommand{\H}{\mathcal{H}}
\newcommand{\D}{\mathcal{D}}
\newcommand{\Lip}{\mathrm{Lip}}
\newcommand{\E}{\mathcal{E}}

\newcommand{\tK}{\widetilde{K}}

\newcommand{\spt}{\mathrm{spt}\,}

\newcommand{\trho}{\widetilde{\rho}}
\newcommand{\tX}{\widetilde{X}}
\newcommand{\tY}{\widetilde{Y}}

\newcommand{\Uf}{U^{\text{free}}}
\newcommand{\uf}{u^{\text{free}}}
\newcommand{\Xf}{X^{\text{free}}}

\newcommand{\Ue}{U^{\text{ex}}}

\newcommand{\Xe}{X^{\text{ex}}}
\newcommand{\Ge}{\Gamma^{\text{ex}}}

\newcommand{\id}{\mathrm{id}}

\newcommand{\vep}{\varepsilon}

\newcommand{\weak}{\DOTSB\protect\relbar\protect\joinrel\rightharpoonup}

\newcommand{\IGNORE}[1]{}

\title{One-dimensional granular system with memory effects}

\author{C. Perrin\footnote{Aix Marseille Univ, CNRS, Centrale Marseille, I2M, Marseille, France; charlotte.perrin@univ-amu.fr}\,, M. Westdickenberg\footnote{Institut f\"ur Mathematik, RWTH Aachen University, Templergraben 55, 52062 Aachen, Germany; mwest@instmath.rwth-aachen.de}}

\begin{document}
	\maketitle

\begin{small}
\begin{center}
{\bf Abstract}
\end{center}
We consider a hybrid compressible/incompressible system with memory
effects, introduced recently by Lefebvre Lepot and Maury for the
description of one-dimensional granular flows. We prove a global
existence result for this system without assuming additional viscous
dissipation. Our approach extends the one by Cavalletti et al. for the
pressureless Euler system to the constrained granular case with memory
effects. We construct Lagrangian solutions based on an explicit
formula using the monotone rearrangement associated to the density. We
explain how the memory effects are linked to the external constraints
imposed on the flow. This result can also be extended to a
heterogeneous maximal density constraint depending on time and space.
	
	\vspace{1cm}
	
\end{small}

\section*{Introduction} In this paper, we consider a
one-dimensional model for immersed granular flows, introduced by
Lefebvre-Lepot and Maury in~\cite{lefebvre2011} and~\cite{maury2007}. The model consists of a system of
nonlinear partial differential equations describing the solid/liquid
mixture through the evolution of the density of solid particles $\rho$
and the Eulerian velocity field $u$. It takes the form
\begin{subnumcases}{\label{eq:granularsystem}}
\partial_t \rho + \partial_x(\rho u) = 0 \label{eq:granularsystem_mass}\\
\partial_t (\rho u) + \partial_x(\rho u^2) + \partial_x p = \rho f \label{eq:granularsystem_mom}\\
\partial_t \gamma + u\partial_x \gamma = -p \label{eq:granularsystem_pot}\\
0 \leq \rho \leq 1 \label{eq:granularsystem_max}\\
(1-\rho)\gamma = 0, \quad \gamma \leq 0, \label{eq:granularsystem_gamma}
\end{subnumcases}
where $p$ represents the pressure, and $f$ is an external force.
Equations~\eqref{eq:granularsystem_mass} and
\eqref{eq:granularsystem_mom} express the local conservation of mass
and momentum, respectively. The density is confined to values between
$0$ (vacuum) and $1$ (for simplicity of presentation), with $\rho=1$
representing the congested state. The pressure $p$ plays the role of
Lagrange multiplier for this pointwise constraint: through the
momentum equation \eqref{eq:granularsystem_mom}, it acts on the fluid
to ensure that condition \eqref{eq:granularsystem_max} remains
satisfied everywhere. The amount of compression that the fluid is
exposed to, but cannot accommodate because of
\eqref{eq:granularsystem_max}, is captured in the \emph{adhesion
potential} $\gamma$, which is linked to the pressure $p$ through
Equation~\eqref{eq:granularsystem_pot}. It expresses a memory effect,
keeping track of the history of the constraint satisfaction of the
system over the course of time.

The system of
Equations~\eqref{eq:granularsystem_mass}--\eqref{eq:granularsystem_gamma}
therefore model two very different regimes that occur in the flow: in
free zones, characterized by the condition $\rho<1$, we have a
pressureless dynamics of a compressible flow. In this regime, both $p$
and $\gamma$ vanish; see \eqref{eq:granularsystem_gamma} and
\eqref{eq:granularsystem_pot}. In the congested zones, characterized
by $\rho=1$, we have the dynamics of an incompressible flow. The
continuity equation \eqref{eq:granularsystem_mass} implies that in
congested zones the velocity must satisfy $\partial_x u = 0$ so that
the action of the external force $f$ must be balanced by the pressure
$p$, which in turn is recorded into the adhesion potential $\gamma$.
Note that in \cites{lefebvre2011, maury2007} instead of 
\eqref{eq:granularsystem_pot} the equation 
\begin{equation}\label{eq:transport-gamma}
\partial_t \gamma + \partial_x (\gamma u) = -p
\end{equation}
is considered, which is \emph{formally} equivalent to
\eqref{eq:granularsystem_pot} since $\partial_x u = 0$ if $\rho=1$. On
the other hand, if $\rho<1$, then $\gamma$ must vanish because of
\eqref{eq:granularsystem_gamma}. We would argue, however, that the
form \eqref{eq:granularsystem_pot} is more natural. In fact,
differentiating this equation with respect to $x$, we obtain
\begin{equation}
	\partial_t \big( \partial_x \gamma \big) + \partial_x\big( u \partial_x \gamma\big) 
		= -\partial_x p.
\label{eq:conser}
\end{equation}
Subtracting this equation from \eqref{eq:granularsystem_mom}, we
observe that the pressure term cancels, giving
\begin{equation}
	\partial_t ( \rho u-\partial_x\gamma) + \partial_x\big( (\rho u-\partial_x\gamma)u \big)
		= \rho f.
\label{E:NEWMOM}
\end{equation}
Hence $\partial_x \gamma$ plays the role of an additional momentum.
Because of the exclusion relation \eqref{eq:granularsystem_gamma}, the
adhesion potential $\gamma$ can only be different from zero where
$\rho=1$, thus $\partial_x \gamma$ is absolutely continuous with
respect to $\rho$. In principle, it is therefore possible to define a
velocity $v$ such that $\partial_x \gamma = \rho v$. We can then
rewrite \eqref{E:NEWMOM} with $w := u-v$ in the form
\begin{equation}
	\partial_t (\rho w) + \partial_x(\rho uw) = \rho f.
\label{E:NEWMOM2}
\end{equation}

In~\cite{maury2007}, the system~\eqref{eq:granularsystem} has been
supplemented with a collision law that prevents elastic shocks between
congested blocks. This condition can be expressed as
\begin{equation}
u(t^+) = \mathrm{P}_{\mathrm{Adm}_{\rho,\gamma}}\big(u(t^-)\big),
\end{equation}
where $u(t^\pm)$ denotes the one-sided limits of the velocity at time
$t$, and with $\mathrm{P}_{\mathrm{Adm}_{\rho,\gamma}}$ the $L^2$
projection onto the set of admissible velocities, defined as the $L^2$
closure of
\[ 
	\left\{ v \in H^1 \colon \begin{cases}
			\partial_x v \geq 0 ~\text{a.e. on}~\{\rho = 1, \gamma=0\}\\
     		\partial_x v = 0 ~\text{a.e. on}~\{\rho = 1, \gamma < 0\}	
		\end{cases} \right\} . 
\]

\smallskip

The memory effects exhibited by system \eqref{eq:granularsystem} have
been brought to light by Maury~\cite{maury2007} in the case of a
single solid particle. He examined a physical system formed by a
vertical wall and a spherical solid particle that is immersed in a
viscous liquid. The particle evolves along the horizontal axis and is
submitted to an external force and to the lubrication force exerted by
the liquid. The latter becomes predominant when the particle is
getting closer to the wall. At first order when the distance $q$
between the particle and the wall goes to $0$, it takes the form
$F_\text{lub} = - C \eta \,\frac{\dot{q}}{q}$, with $\eta$ the
viscosity of the liquid, $C>0$ a constant that depends on the diameter
of the particle. It prevents the contact in finite time of particle
and wall.

Considering the limit of vanishing liquid viscosity $\eta=\vep
\rightarrow 0$, Maury proved in~\cite{maury2007} the convergence
toward a hybrid system (see system~\eqref{eq:glueyparticle} below)
describing the two possible states of the system: \emph{free} when
$q>0$ and \emph{stuck} when $q=0$. In the limit, the system involves a
new variable $\gamma$, the adhesion potential, which is the residual
effect of the singular lubrication force $F^\vep_\text{lub}$ for
$\vep\rightarrow 0$. The potential describes the stickyness of the
particle: even in case of a pulling external force, it may take some
time before the particle takes off from the wall.
  
Lefebvre-Lepot and Maury have extended this idea to a one-dimensional
macroscopic system of aligned solid particles:
system~\eqref{eq:granularsystem} is obtained in~\cite{lefebvre2011} as
the formal limit of
\begin{subnumcases}{\label{eq:suspension}}
\partial_t \rho_\vep + \partial_x(\rho_\vep u_\vep) = 0 \\
\partial_t(\rho_\vep u_\vep) + \partial_x(\rho u_\vep^2) - \partial_x\left(\dfrac{\vep}{1-\rho_\vep}\partial_x u_\vep\right) = \rho_\vep f.
\end{subnumcases}
The lubrication force is represented at this macroscopic scale by the
singular viscous term $\partial_x(\frac{\vep}{1-\rho_\vep}\partial_x
u_\vep)$, which prevents, by analogy with the single particle case,
the formation of congested domains $\rho=1$ when $\vep >0$. The
rigorous proof, however, of the convergence of solutions
of~\eqref{eq:suspension} to solutions of~\eqref{eq:granularsystem}
remains an open problem. The mathematical difficulty of this singular
limit relies in the lack of compactness of the non-linear term
$\rho_\vep u_\vep^2$. This kind of singular limit has nevertheless
been proved in~\cite{perrin20161d} (see also~\cite{perrin2016AMRX} for
a result in dimension 2) on an augmented system where an additional
physical dissipation is taken into account.

\bigskip
In this paper, we want to study the
system~\eqref{eq:granularsystem} directly, without any lubrication
approximation. 
For that purpose, we take advantage of the link between
model~\eqref{eq:granularsystem} and the model of pressureless gas
dynamics in one space dimension:
\begin{subnumcases}{\label{eq:pressureless}}
\partial_t \rho + \partial_x(\rho u) = 0 \label{eq:pressureless_mass} \\
\partial_t(\rho u) + \partial_x(\rho u^2) = 0. \label{eq:pressureless_mom}
\end{subnumcases} 
We establish a global existence result for
weak solutions to the following system:
\begin{subnumcases}{\label{eq:pressurelessMOD}}
	\partial_t \rho + \partial_x(\rho u) = 0 
\label{eq:pressurelessMOD_mass}\\
	\partial_t(\rho u-\partial_x\gamma) 
		+ \partial_x\big( (\rho u-\partial_x\gamma)u \big) = \rho f 
\label{eq:pressurelessMOD_mom}\\
	0 \leq \rho \leq 1  
\label{eq:pressurelessMOD_rho}\\
	(1-\rho)\gamma = 0, \quad \gamma \leq 0. 
\label{eq:pressurelessMOD_gamma}
\end{subnumcases}

Among the large literature that exists for the pressureless
system~\eqref{eq:pressureless}, we are interested in the recent
results of Natile and Savar\'e~\cite{natile2009} and Cavalletti et
al.~\cite{cavalletti2015} that develop a Lagrangian approach based on
the representation of the density $\rho$ by its monotone rearrangement
$X$, which is the optimal transport between the Lebesgue measure
$\mathcal{L}^1_{|[0,1]}$ and $\rho$; see~\cite{santambrogio2015}.

\bigskip

Let us also mention that the granular system~\eqref{eq:granularsystem}
can be seen as a non-trivial extension of the pressureless Euler
equations under maximal density constraint
\begin{subnumcases}{\label{eq:stickyblocks}}
\partial_t \rho + \partial_x(\rho u) = 0 \\
\partial_t(\rho u) + \partial_x(\rho u^2) + \partial_x \pi = 0 \\
0 \leq \rho \leq 1 \\
(1-\rho) \pi = 0 , \quad \pi \geq 0.
\end{subnumcases}
This system has been first introduced by Bouchut et al.
\cite{bouchut2000} as a model of two-phase flows and then studied by
Berthelin~\cite{berthelin2002} and Wolansky~\cite{wolansky2007}. The
results rely on a discrete approximation generalizing the sticky
particle dynamics used for the pressureless system. Recently,
numerical methods based on optimal transport tools have been developed
for this system; see~\cites{maury2015, preux2016}.

For viscous fluids, i.e., Navier-Stokes systems, a theoretical
existence result can be found in~\cite{perrin2015} in the case where
the maximal density constraint $\rho^*(x)$ is a given function of the
space variable. Recently, Degond et al. have proved
in~\cite{degond2016} the existence of global weak solutions to the
Navier-Stokes system with a time and space dependent maximal
constraint $\rho^*(t,x)$ that is transported by the velocity $u$: it
satisfies the transport equation
\begin{equation}\label{eq:transport_rho*}
\partial_t \rho^* + u \, \partial_x\rho^* = 0.
\end{equation}
Numerical simulations are have been studied
in~\cites{degond2016, degond2017} with applications to crowd dynamics.
This type of heterogeneous maximal constraint may be also relevant for
the dynamics of floating structures; see for instance
Lannes~\cite{lannes2016}.

\medskip

The paper is organized as follows:

\medskip

In Section~\ref{sec:pressureless} we briefly review the literature on
the pressureless gas dynamics and introduce the mathematical tools
linked to a Lagrangian description. In Section~\ref{sec:mainresult} we
explain formally how these tools can be extended to the
system~\eqref{eq:pressurelessMOD} and give our main
existence result. Section~\ref{sec:proof} is devoted to the proof of
this result and Section~\ref{sec:numerics} presents some numerical
simulations. In the last section, we extend finally the result to the
special case of time and space dependent maximal density constraint
that satisfies the transport equation~\eqref{eq:transport_rho*}.

\section{Lagrangian approach for the pressureless Euler
equations}{\label{sec:pressureless}} The pressureless gas dynamics
equations, augmented by the assumption of adhesion dynamics, has been
proposed as a simple model for the formation of large scale structures
in the universe such as aggregates of galaxies. It is linked to the
sticky particle system introduced by Zeldovich
in~\cite{zeldovich1970}. The work of Bouchut~\cite{bouchut1994}
highlights the obstacles to proving existence of classical solutions
to~\eqref{eq:pressureless} (concentration phenomena on the density,
lack of uniqueness under classical entropy conditions). Since then,
several different mathematical approaches have been proposed in the
literature for proving the global existence of measure solutions under
suitable entropy conditions (see again~\cite{bouchut1994}), among
which there are approximations by the discrete sticky particles
dynamics \cites{brenier1998, natile2009}, approximation by viscous
regularization \cites{sobolevskii1997, boudin2000} or, more recently,
derivation by a hydrodynamic limit \cite{jabin2016}.

\medskip In particular, Natile and Savaré use \cite{natile2009} an
interesting Lagrangian characterization of the density $\rho$ by its
monotone rearrangement $X$ to show convergence of the discrete sticky
particle system as the number $N$ of particles goes to $+\infty$. To
every probability measure $\rho \in \mathcal{P}_2(\R)$ (i.e., with
finite quadratic moment $\int_\R |x|^2 \,\rho(dx) < +\infty$) there
is associated a unique transport $X \in K$, the closed convex cone of
non-decreasing maps in $L^2(0,1)$, such that
\begin{equation}
\rho_t = (X_t)_\# \mathcal{L}^1_{|[0,1]}.
\end{equation}
Here $\mathcal{L}^1_{|[0,1]}$ is the one-dimensional Lebesgue measure
restricted to the interval $[0,1]$ and $\#$ denotes the
\emph{push-forward} of measures, defined for all Borel maps
$\zeta\,:\,\R \rightarrow [0,\infty]$ by
\begin{equation}
\int_\R \zeta(x) \,\rho_t(dx) = \int_0^1 \zeta\big(X_t(y)\big) \,dy.
\end{equation}
If now $(\rho, u)$ is a solution in the distributional sense
of~\eqref{eq:pressureless}, then $u_t$ can be associated to the
Lagrangian velocity $U_t := \dot{X}_t$ (in the sequel all the
Lagrangian variables will be denoted by capital letters and the
Eulerian ones by the corresponding small letters) through
\begin{equation}
U_t(y) = u_t\big( X_t(y) \big).
\end{equation}
In~\cite{natile2009}, Natile and Savaré show different
characterizations of the transport $X$ associated to an Eulerian
solution of~\eqref{eq:pressureless}, in particular they prove that
\begin{equation}\label{eq:projection_pressureless}
X_t = \mathrm{P}_K\big(\bar{X} + t \bar{U}\big)\quad \text{for all $t \geq 0$,}
\end{equation} 
where $ \mathrm{P}_K$ is the $L^2(0,1)$ projection onto
the closed convex set $K$
and $\bar{X},\bar{U}$ are respectively the monotone rearrangement and the Lagrangian velocity associated to the initial data. The map $\bar{X} + t \bar{U}$ represents the free
motion path, which is at the discrete level the transport
corresponding to the case where the particles do not interact at all.

These arguments have been extended by Brenier et al.
\cite{brenier2013} to systems including an interaction between the
discrete particles. This interaction is represented at the continuous
level by a force $f(\rho)$ in the right-hand side of the momentum
equation~\eqref{eq:pressureless_mom}.

\bigskip

Recently, Cavalletti et al.~\cite{cavalletti2015} have taken advantage
of the formula~\eqref{eq:projection_pressureless} to construct
directly global weak solutions to~\eqref{eq:pressureless} without any
discrete approximation by sticky particles. To this end, they define
for all positive times $t$ the transport $X_t$, associated to an
initial data $(\bar{\rho},\bar{u})$, by
equation~\eqref{eq:projection_pressureless}. The Lagrangian variables
$\bar{X}$ and $\bar{U}$ are defined by
\begin{equation} \label{eq:initial_pessureless}
\bar{\rho} = (\bar{X})_\# \mu, \quad \bar{U} := \bar{u} \circ \bar{X} 
\end{equation}
for a more general reference measure $\mu$ in $\mathcal{P}_2(\R)$ (for
instance $\mu = \bar{\rho}$ and in this case $\bar{X} = \id$). As a consequence
of the contraction property of the projection operator $\mathrm{P}_K$,
the map $t\mapsto X_t$ is Lipschitz continuous and thus differentiable
for a.e. $t$, which allows us to define the Lagrangian velocity $U_t
:= \dot{X}_t$. Cavalletti et al. \cite{cavalletti2015} introduce the
subspace in $L^2(\R,\mu)$ formed by functions which are essentially
constant where $X_t$ is constant:
\begin{equation}
\mathcal{H}_{X_t} = \text{$L^2(\R,\mu)$-closure of} ~\big\{ \varphi \circ X_t
	\colon \varphi \in \mathcal{D}(\R) \big\}.
\end{equation}
This space is a subset of the tangent cone to $K$ at $X_t$, denoted by
$\mathbb{T}_{X_t}K$, in which the Lagran\-gian velocity is contained.
One can then show that $U_t$ is the orthogonal projection of $\bar{U}$
onto the space $\mathcal{H}_{X_t}$:
\begin{equation}\label{eq:projU_pressureless}
U_t = \mathrm{P}_{\mathcal{H}_{X_t}}(\bar{U}).
\end{equation}
This property ensures that there exists, for a.e. $t$, an Eulerian
velocity $u_t \in L^2(\R,\rho_t)$ with the property that $U_t = u_t \circ
X_t$. This is the key argument for recovering the weak formulations of
the gas dynamics
equations~\eqref{eq:pressureless_mass}--\eqref{eq:pressureless_mom} in
the Eulerian formulation.

\bigskip

By comparison, our granular system written under the pressureless
form~\eqref{eq:pressurelessMOD} involves an additional maximal density
constraint $\rho \leq 1$, an additional variable $\gamma$ linked to
this maximal constraint, and an external force $f$. 
We explain in the next section how
to extend the previous tools in order to deal with these additional
constraints and variables.

\section{Extension to granular flows, main result}{\label{sec:mainresult}}

Before announcing our existence result, we need to explain how to
adapt the Lagrangian tools mentioned above when an external force and
a maximal density constraint is given. A good way to do this is to
come back to the microscopic approach, by nature Lagrangian, developed
by Maury in~\cite{maury2007} for a single sticky particle in contact
with a wall.

\paragraph{Single particle case.} Maury \cite{maury2007} proves by a
vanishing viscosity limit (viscosity of liquid in which the particle
is immersed), the existence of solutions to the hybrid system
\begin{subnumcases}{\label{eq:glueyparticle}}
\dot{q} + \gamma = \bar{u} + \int_0^t{f(s) \,ds}\label{eq:glueyparticle_dyn} \\
q \geq 0, \quad \gamma \leq 0, \quad q\,\gamma = 0, 
\end{subnumcases}
which describes the two possible states of the system: \emph{free}
when $q>0$ (that is, the particle evolves freely under the external
force $f$), and \emph{stuck} whenever $q=0$. In this latter case, the
adhesion potential is activated and is equal to
\begin{equation}\label{eq:micro_ufree}
\uf(t) = \bar{u} + \int_0^t{f(s) \,ds},
\end{equation}
which is the velocity the particle would have if there was no wall on
its trajectory. System~\eqref{eq:glueyparticle} is in fact equivalent
to the following second order system (see~\cite{lefebvre2007})
\begin{subnumcases}{\label{eq:gluey2nd}}
\ddot{q} = f + \lambda \\
\dot{q}(t^+) = \mathrm{P}_{C_{q,\gamma}(t)}\dot{q}(t^-) \\
\spt(\lambda) \subset \{t \colon q(t) = 0 \} \\
\dot{\gamma} = - \lambda  \label{eq:potmicro}\\
q \geq 0, \quad \gamma \leq 0,
\end{subnumcases}
where
$C_{q,\gamma}(t)$ denotes the set of admissible velocities
\[
C_{q,\gamma}(t) = \begin{cases}
~ \{0\} \quad & \text{if} \quad \gamma(t^-) < 0 \\
~ \R^+  \quad & \text{if} \quad \gamma(t^-) = 0,~ q(t)=0 \\
~ \R    \quad & \text{otherwise.}
\end{cases}
\]
It ensures that the particle cannot cross the wall and that it sticks
to the wall as long as $\gamma <0$. By comparison with
system~\eqref{eq:granularsystem}, an analogy can be made between the
variables $q$ and $1-\rho$, between $\lambda$ and $-\partial_x p$ and
thus between $\gamma_\text{micro}$ defined by~\eqref{eq:potmicro} and
$-\partial_x \gamma_\text{macro}$.

\paragraph{Extension of the Lagrangian approach.} Let $\rho_0 \in
\mathcal{P}_2(\R)$ denote the initial density. We assume that $\rho_0$
is absolutely continuous with respect to the Lebesgue measure and that
its density (also denoted $\rho_0$, for simplicity) satisfies the
maximal constraint
\begin{equation}
0 \leq \rho_0 \leq 1 \quad \text{a.e.}
\end{equation}
As suggested by Cavalletti et al.~\cite{cavalletti2015} we set $\bar{\rho} =
\rho_0$ as reference measure. 

\begin{definition}
The set of square-inte\-grable functions with respect to the measure
$\bar{\rho}$ will be denoted $L^2(\R, \bar{\rho})$. Let $\langle \cdot,\cdot
\rangle$ be the induced inner product. The space of $p$-inte\-grable
functions on the domain $\Omega$ for the Lebesgue measure will be
denoted by $L^p(\Omega)$.
\end{definition}

\begin{framed}
\noindent In the following, we will switch freely between absolutely
continuous measures $\rho(dx)$ and their Lebesgue densities $\rho(x)
\,dx$. The meaning will be clear from the context.
\end{framed}

\noindent{\it Set of admissible transports.} To each $\rho_t$ we
associate a monotone transport map $X_t$ through
\begin{equation}
\rho_t = (X_t)_\# \bar{\rho}.
\label{eq:rhoxt}
\end{equation}
To express the maximal density constraint $\rho_t \leq 1$ in terms of
a constraint on the transport map $X_t$, we consider a maximally
compressed density with the same total mass as $\bar{\rho}$, which is a
characteristic function of some interval $\widetilde{I}$ of length
one. For definiteness, we assume this interval to be centered around
the center of mass of $\bar{\rho}$, but the construction is invariant
under translation since constants can be absorbed into the transport
map. Let $\trho$ be the probability measure associated to the
characteristic function of $I_0$. Let $\tX$ be the unique
nondecreasing transport map in $L^2(\R,\bar{\rho})$ (see
\cite{santambrogio2015} Theorem 2.5, for example) such that
\begin{equation}\label{eq:cong_transport} 
\trho = \tX_\# \bar{\rho}. 
\end{equation}
The push forward formula implies that $\partial_y \tX(y) > 0$ for
$\bar{\rho}$-a.e. $y\in\R$ and
\begin{equation}
	\trho(x) = \frac{\bar{\rho}(\tX^{-1}(x))}{\partial_y \tX(\tX^{-1}(x))}
	\quad\text{for a.e. $x\in I_0$;}
\label{E:EUQ}
\end{equation}
see \cite{ambrosio2008} Lemma 5.5.3. In particular, we have
$\bar{\rho}(y)/\partial_y \tX(y) = 1$ for $\bar{\rho}$-a.e. $y\in \R$. The
measure $\rho_t$ in \eqref{eq:rhoxt} is absolutely continuous with
respect to the Lebesgue measure if and only if the approximate
derivative $\partial_y X_t(y)>0$ for $\bar{\rho}$-a.e. $y\in\R$. For
$\rho_t$-a.e. $x\in\R$, we then have
\begin{equation}
\rho_t (x) = \dfrac{\bar{\rho}(X_t^{-1}(x))}{\partial_yX_t(X_t^{-1}(x))}
= \dfrac{\bar{\rho}(X_t^{-1}(x))}{\partial_y\tX(X_t^{-1}(x))} \,\dfrac{\partial_y\tX(X_t^{-1}(x))}{\partial_yX_t(X_t^{-1}(x))}
= \dfrac{\partial_y\tX(X_t^{-1}(x))}{\partial_yX_t(X_t^{-1}(x))}.
\label{eq:fract}
\end{equation}
In order to guarantee the maximal density constraint we are thus led
to consider transport maps $X_t$ such that $\partial_y\tX(y) \leq
\partial_yX_t(y)$ for $\bar{\rho}$-a.e. $y\in\R$. We therefore introduce
the closed convex set of admissible transports maps in $L^2(\R,
\bar{\rho})$ as follows: We define
\begin{equation}
\tK := K + \tX,
\label{eq:kol}
\end{equation}
where $K$ is the cone of monotone (more precisely: non-decreasing)
maps of $L^2(\R, \bar{\rho})$. To the transport map $X_t \in \tK$, we
associate the monotone transport map
\begin{equation}
S_t := X_t -\tX \in K.
\end{equation} 
Note that monotone maps are differentiable a.e., with
nonnegative derivative.

\begin{remark} Coming back to the definition of $\trho$, we
observe that the position of the interval $\widetilde{I}$ does not
matter for the definition of $\tK$ since the translations can be
absorbed in $K$.
\end{remark}

\noindent{\it Formal description of the dynamics.} In
order to define for all times $t$ an appropriate transport map $X_t$,
we need to extend the notion of free transport $\bar{X} + t \bar{U}$
used in~\eqref{eq:projection_pressureless} to the case where the external
force $f$ is applied on the system. In particular, we need to extend
the notion of free velocity, which for the pressureless system is
simply the initial velocity $\bar{U}$. In our case, inspired by the
microscopic case~\eqref{eq:micro_ufree}, we are naturally led to set
\begin{equation}\label{eq:freevelocity}
\Uf_t := \bar{U} + \int_0^t{f(s,X_s) \,ds},
\end{equation} 
integrating along the trajectories $t\mapsto X_t(y)$ starting at $y$.
The free trajectory at time $t$ would be then be given by
the formula
\begin{equation}
\Xf_t := \bar{X} + \int_0^t{\Uf_s \, ds},
\end{equation}
and in analogy with \eqref{eq:projection_pressureless}, we consider
\begin{equation}\label{df:Xt}
X_t := \mathrm{P}_{\tK}(\Xf_t) =  \mathrm{P}_{\tK}\left(\bar{X} + \int_0^t{\Uf_s \, ds}\right);
\end{equation}
see~\cite{brenier2013} for a similar formulation. Note
carefully that equations~\eqref{df:Xt}--\eqref{eq:freevelocity} form a
coupled system insofar as the free velocity $\Uf_t$ depends on the
trajectory $X_s$ itself. Establishing the existence and uniqueness of a
solutions is non trivial and requires suitable assumptions on the
external force $f$. We detail this point in
Lemma~\ref{lem:existence-Xt} below.
The associated velocity, formally
defined as $U_t = \frac{d}{dt} X_t$, not only has to belong the
tangent cone to $K$ at $S_t = X_t - \tX$, defined as
\begin{equation}\label{df:tangent_cone}
	\mathbb{T}_{S_t} K := \text{$L^2(\R,\bar{\rho})$-closure of $T_{S_t} K$}
	\quad \text{where} \quad 
	T_{S_t} K := \bigcup_{h>0} h \big(K - S_t \big),
\end{equation}
it  has to be constant on each congested block for a.e. $t$. That is, $U_t$
must belong to the set
\begin{equation}\label{df:H_S}
	\mathcal{H}_{S_t} := \left\{ U \in L^2(\R,\bar{\rho}) \colon
		\text{$U$ is a.e. constant on maximal intervals in $\Omega_{S_t}$} \right\}
\end{equation}
where $\Omega_{S_t}$ is the union of all non-trivial
intervals on which $S_t$ is constant. In analogy to the microscopic
case~\eqref{eq:glueyparticle_dyn}, we define an adhesion potential
$\Gamma_t$; see \eqref{E:DEFGMM}.

\paragraph{Definition of weak solutions and main result.} Here is our
solution concept.

\begin{definition}\label{D:SOLNN}
Given suitable initial data $(\bar{\rho},\bar{u})$, a
triple $(\rho, u, \gamma)$ is called a weak solution of
system~\eqref{eq:pressurelessMOD} provided that
\begin{itemize}
\item $(\rho, u, \gamma)$ satisfies
\begin{gather}
	\rho_t \in \mathcal{P}_2(\R),
	\;
	u_t \in L^2(\R, \rho_t)
	\quad \text{for a.e. $t\in[0,T]$},
\label{eq:one}\\
	\gamma \in L^\infty\big( [0,T];W^{1,1}(\R) \big);
\label{eq:two}
\end{gather}
\item the density constraint~\eqref{eq:pressurelessMOD_rho} and the
  exlusion principle~\eqref{eq:pressurelessMOD_gamma} hold almost
  everywhere;
\item equations~\eqref{eq:pressurelessMOD_mass}
  and~\eqref{eq:pressurelessMOD_mom} are satisfied in the sense of
  distributions:
\begin{align}
	& \int_0^T\int_\R \Big(\partial_t \xi(t,x) + u_t(x) \partial_x \xi(t,x) \Big)
		\rho_t(x) \,dx \,dt 
\label{eq:weak_mass}\\
	& \qquad\qquad
		= - \int_{\R}{\xi(0,x) \, \bar\rho(x) \,dx} 
	\quad\text{for all $\xi \in \mathcal{C}^\infty_c\big([0,T) \times \R \big)$,} 
\nonumber\\
	& \int_0^T\int_\R \Big(\partial_t \varphi(t,x) + u_t(x) \partial_x \varphi(t,x)\Big)
		\big(\rho_t(x)u_t(x)-\partial_x \gamma_t(x)\big) \,dx \,dt 
\label{eq:weak_momentum}\\
	& \qquad\qquad
		+ \int_0^T\int_\R \varphi(t,x) \,\rho_t(x) f_t(x) \,dx \,dt
\nonumber\\
	& \qquad\qquad
		= - \int_\R \varphi(0,x) \, \bar\rho(x) \bar u(x) \,dx
	\quad\text{for all $\varphi \in \mathcal{C}^\infty_c\big([0,T) \times \R \big)$.}
\nonumber
\end{align}
\end{itemize}
\end{definition}

\begin{remark}\label{E:INVELO}
As explained before, the velocity must be compatible with the flow
configuration in the sense that it is non-decreasing on congested
blocks; recall \eqref{df:H_S}. We can write
\[
	\bar\rho(x) \bar u(x) = \bar\rho(x) u_0(x)-\partial_x\gamma_0(x),
\]
where $u_0$ is the $L^2(\R,\bar\rho)$-projection of $\bar u$ onto the
tangent cone $\mathbb{T}_{S_0} K$ and the adhesion potential
$\gamma_0$ is defined in analogy to \eqref{E:DEFGMM}.
\end{remark}

\begin{theorem}	
Let $T > 0$ and external force $f \in L^\infty\big(0,T;\Lip(\R)\cap
L^\infty(\R)\big)$ be given. Suppose that $\bar\rho \in
\mathcal{P}_2(\R)$ with $\bar\rho \ll \mathcal{L}^1$ and $0 \leq \bar\rho
\leq 1$ a.e., and that $\bar u \in L^2(\R,\bar\rho)$. Define
\begin{equation}
	\bar X := \id, \; \bar{U} := \bar u, 
	\quad\text{so that}\quad
	X_0=X_0^\text{free}=\id, \; \Uf_0=\bar u, \; \rho_0 := \bar\rho.
\label{eq:initiald}
\end{equation}
There exists a curve $[0,T] \ni t \mapsto X_t \in \tK$ that is
differentiable for a.e. $t \in (0,T)$ and solves the coupled system of
equations \eqref{eq:freevelocity}--\eqref{df:Xt}. The
following quantities are well-defined:
\begin{equation}
	U_t(y) := \dot{X}_t(y), \; 
	\Gamma_t(y) := \int_{-\infty}^{y}{\Big(U_t(z) - \Uf_t(z)\Big) \bar{\rho}(z) \,dz} 
\label{E:DEFGMM}
\end{equation}
for $y\in\R$ and a.e. $t \in (0,T)$. There exist $(u_t,\gamma_t)
\in\mathcal{L}^2(\R, \rho_t)\times W^{1,1}(\R)$, such that
\[ 
	U_t = u_t\circ X_t, \; \Gamma_t = \gamma_t \circ X_t
	\quad\text{where}\quad
	\rho_t := (X_t)_\# \bar{\rho}.
\]
The triple $(\rho, u, \gamma)$ is a global weak solution of
system~\eqref{eq:pressurelessMOD}.
\end{theorem}

\begin{remark} The assumption $f \in
L^\infty\big(0,T;\Lip(\R)\cap L^\infty(\R)\big)$ can be relaxed to
include a larger class of forces. We will stick to it here to simplify
the presentation.

Notice that $X_0 = \id $ satisfies \eqref{eq:fract} for $t=0$,
hence $X_0 \in \tK$ as expected.
\end{remark}

\section{Construction of global weak solutions}{\label{sec:proof}}

Our proof consists of three steps. First we establish
existence and uniqueness of $X_t$ (defined by~\eqref{df:Xt}) and
$U_t$. We will prove that the velocity $U_t$ is admissible in the
sense that it belongs to the set $\mathcal{H}_{S_t}$ defined
in~\eqref{df:H_S}. Introducing next the adhesion potential as
in~\eqref{E:DEFGMM}, we show in Subsection~\ref{sec:Gamma} that it is
non-positive and supported in the congested domain. Finally, we check
in Subsection~\ref{sec:weak_form} that the Eulerian variables
$(\rho,u,\gamma)$ associated to~$(X_t,U_t,\Gamma_t)$ with $t\in[0,T]$
satisfy the weak
formulations~\eqref{eq:weak_mass}--\eqref{eq:weak_momentum} of
system~\eqref{eq:pressurelessMOD}.

\subsection{Definition of the transport and velocity} Let
us begin by justifying the fact that we can define in a unique manner
$X_t$ for all times.

\begin{lemma}\label{lem:existence-Xt}
For all $t\in [0,T]$, there exists a unique solution $(X_t, \Uf_t)$
to~\eqref{eq:freevelocity}--\eqref{df:Xt}.
\end{lemma}

\begin{proof}
Let $\E := \mathcal{C}([0, T], L^2(\R,\bar\rho))$ endowed with the norm
\[ 
	\|X\|_{\E} = \max_{t \in [0,T]} e^{-2 \sqrt{k}t} \|X_t\|_{L^2(\R, \bar{\rho})} 
\]
where $k$ is the Lipschitz constant of the external force $f$.
We define a map $\mathcal{T}$ by
\[ 
	\mathcal{T}(X)(t) := \mathrm{P}_{\tK}\left(\bar X + t \bar{U} 
		+ \int_0^t \int_0^\tau {f(s,X_s) \, ds \, d\tau} \right), \quad t\in [0,T],
\]
for all $X \in \E$. To prove the existence of a unique solution
to~\eqref{eq:freevelocity}--\eqref{df:Xt} we will show that the map
$\mathcal{T}$ is a contraction on $\E$. Consider $X^1, X^2 \in \E$
starting at $t= 0$ from $\bar X$ with velocity $\bar{U}$. Thanks to the contraction property of the projection map we
have
\begin{align*}
\| \mathcal{T}(X^1)(t) - \mathcal{T}(X^2)(t) \|_{L^2(\R, \bar{\rho})}
& \leq \left\| \int_0^t\int_0^\tau{\big(f(s,X^1_s) -f(s,X^2_s) \big) \,ds \, d\tau} \right\|_{L^2(\R, \bar{\rho})} \\
& \leq \int_0^t\int_0^\tau { \|f(s,X^1_s) -f(s,X^2_s) \|_{L^2(\R, \bar{\rho})} \,ds \,d\tau } \\
& \leq k \int_0^t\int_0^\tau { \|X^1_s - X^2_s \|_{L^2(\R, \bar{\rho})} \,ds \,d\tau } \\
& \leq k \|X^1 - X^2 \|_{\E} \int_0^t\int_0^\tau { e^{2\sqrt{k}s} \,ds \,d\tau  } \\
& \leq \dfrac{1}{4} e^{2\sqrt{k}t} \|X^1 - X^2 \|_{\E}
\end{align*}
for all $t \in [0,T]$. We have therefore
\[  
	\| \mathcal{T}(X^1) - \mathcal{T}(X^2) \|_{\E} \leq \dfrac{1}{4} \|X^1 - X^2 \|_{\E}.  
\]
Applying the Banach Fixed Point Theorem, we then conclude that there
exists a unique map $t \mapsto X_t$ solution of~\eqref{df:Xt} as well
as a unique $\Uf_t$ for all times.
\end{proof}

We now recall two useful lemmas proved in~\cite{cavalletti2015}.

\begin{lemma}[Lemma 3.1~\cite{cavalletti2015}]\label{L:ONE}
For given $S \in L^2(\R, \bar{\rho})$ monotone, define $\Pi_{S} :=
(\id, S)_\# \bar{\rho}$. Then there exists a Borel set $N_{S}$ such that
$\bar{\rho}(N_{S}) = 0$ and
\[
	(y, S(y)) \in \spt \Pi_{S} \quad \text{for all}~ y \in \R \setminus N_{S}. 
\]
\end{lemma}

\begin{lemma}[Steps 1 and 2 of Lemma 3.7~\cite{cavalletti2015}]\label{lem:Ot}
Assume that $N_S$ is the $\bar{\rho}$-null set associated to a monotone
map $S \in L^2(\R,\bar{\rho})$, as introduced in the previous lemma. Let
\begin{gather*} 
	L^z := \left\{ y \in \R \setminus N_S \colon S(y) = z \right\},
\\
	\mathcal{O} := \left\{ z\in \R \colon \text{$L^z$ has more than one element} \right\}.
\end{gather*}
The set $\mathcal{O}$ is at most countable and $S$ is injective on $\R
\setminus \bigcup_{z\in \mathcal{O}} L^z$.
\end{lemma}

\begin{remark}
Lemma~\ref{L:ONE} shows that the support of the transport
plan $\Pi_S$ induced by a map $S$ is supported on the graph of $S$
(which is a subset of the product space $\R\times\R$) up to a
negligible set. This technical fact will be needed in the proof of
Proposition~\ref{prop:space_velocity} below.
\end{remark}

\begin{remark}\label{R:EXPLAN}
To understand Lemma~\ref{lem:Ot}, recall that we are considering
maps $X_t$ in the cone $\tK := K + \tX$, where $K$ denotes
the cone of non-decreasing maps of $\L^2(\R, \bar{\rho})$ and $\tX$ is a
fixed monotone map. Because of
\eqref{eq:fract}, the density $\rho_t := (X_t)_\# \bar{\rho}$ satisfies
the constraint
\[
	\rho_t(x) = 1
	\quad\Longleftrightarrow\quad
	\partial_y\tX_t(y) = \partial_yX_t(y)
\]
with $x=X_t(y)$. This is equivalent to the condition $\partial_y
S_t(y) = 0$ where $S_t := X_t - \tX$. Recall that $S_t$ is
non-decreasing. We are thus led to consider points where $S_t$ is
constant on some open neighborhood. Applying Lemma~\ref{L:ONE} with $S
\equiv S_t$ and denoting by $L^z_t, \mathcal{O}_t$ the corresponding
sets defined above, we observe that these sets are precisely given by
$L^z_t$, provided $L^z_t$ has more than one point. By monotonicity of
$S_t$, any such $L^z_t$ must be an interval. There are at most
countably many. For any such $z$ we have $X_t(y) = z+\tX(y)$ for a.e.
$y\in L^z_t$, thus
\begin{equation}
	\rho_t(x) = 1
	\quad\text{for a.e. $x \in \big\{ z+\tX(y) \colon y\in L^z_t \big\}$.}
\label{eq:congz}
\end{equation}
This defines one congested zone. Note that $\partial_y\tX(y)>0$ for
$\bar{\rho}$-a.e. $y\in\R$ so that $\tX$ is strictly increasing. The
congested zone defined in \eqref{eq:congz} has positive length since
$L^z_t$ contains an open interval. Consequently, there can be at most
countably many congested zones. Let
\[
	\Omega_{S_t} := \bigcup_{z\in\mathcal{O}_t} \big\{ z+\tX(y) \colon y\in L^z_t \big\}.
\]
\end{remark}

\begin{proposition}\label{prop:space_velocity}
The velocity $U_t := \frac{d}{dt}X_t$ exists and belongs to
$\H_{S_t}$ for a.e. $t \in (0,T)$.
\end{proposition}

\begin{proof}
Due to the contraction property of the projection, we have
\begin{align*}
\| X_{t+h} - X_t \|_{L^2(\R, \bar{\rho})} 
	& \leq \left\|\int_0^t{\Uf_s \,ds}\right\|_{L^2(\R, \bar{\rho})}  \\ 
	& \leq h \|\bar{U}\|_{L^2(\R, \bar{\rho})} 
		+ \left\| \int_t^{t+h}{\left(\int_0^s{f(\tau,X_\tau) \,d\tau}\right)\,ds}\right\|_{L^2(\R, \bar{\rho})}
\end{align*}
and since $f\in L^\infty\big(0,T;L^\infty(\R)\big)$ we deduce that 
\begin{equation}\label{eq:Xt_lipschitz}
\| X_{t+h} - X_t \|_{L^2(\R, \bar{\rho})} 
	\leq |h| \Big( \|\bar{U}\|_{L^2(\R, \bar{\rho})} + C\big(\|f\|_{L^\infty}\big) \Big).
\end{equation}
This proves that $t \mapsto X_t$ is Lipschitz continuous. Its time-derivative exists strongly  and
\[ U_t = \lim_{h \rightarrow 0^+} \dfrac{X_{t+h}-X_t}{h} = -\lim_{h \rightarrow 0^+} \dfrac{X_{t-h}-X_t}{h} \]
for a.e. $t\in (0,T)$.
We deduce that
\[ U_t \in \mathbb{T}_{S_t}K \cap \big(-\mathbb{T}_{S_t}K \big).\]
from the definition \eqref{df:tangent_cone} of tangent cone, now $U_t \in \mathbb{T}_{S_t}K$ implies that there exist two sequences $(W_t^k),\, (\lambda^k)$ with $W_t^k \in K$ and $\lambda^k > 0$, such that
\[ U_t^k = W_t^k -\lambda^k S_t \quad \text{converges strongly to $U_t$ in $L^2(\R, \bar{\rho})$.} \] 
We can then extract a subsequence, still denoted $(U_t^k)$, that converges a.e. towards $U_t$.
For every $k$ we denote by $N^k$ the $\bar{\rho}$-null set associated to $W_t^k$, as introduced in Lemma~\ref{L:ONE}.
There exists a $B \subset \R$ with $\bar{\rho}(B)=0$, such that $\displaystyle \bigcup_k N^k \subset B$ and
\[ 
	U_t^k(y) \longrightarrow U_t(y) 
	\quad\text{for all $y \in \R \setminus B$, as $k\rightarrow\infty$.}
\]
For all $x\in \mathcal{O}_t$ and $y_1,y_2 \in L_t^x \setminus B$ (see Remark~\ref{R:EXPLAN} for notation), we have
\begin{align*}
(y_1-y_2)\big(U_t^k(y_1) -U_t^k(y_2)\big)
& = (y_1-y_2) \big( W_t^k(y_1) -\lambda^k S_t(y_1) -  W_t^k(y_2) + \lambda^k S_t(y_2) \big) \\
& = (y_1-y_2) \big( W_t^k(y_1) - W_t^k(y_2) \big)\\
& \geq 0,
\end{align*}
by monotonicity of $W_t^k$, and thus by passing to the limit $k \rightarrow +\infty$
\[ (y_1-y_2)\big(U_t(y_1) -U_t(y_2)\big) \geq 0. \]
Using now the fact that $U_t \in \big(-\mathbb{T}_{S_t}K \big)$, we obtain in the same way
\[ (y_1-y_2)\big(U_t(y_1) -U_t(y_2)\big) \leq 0. \]
Thus 
\begin{equation}
(y_1-y_2)\big(U_t(y_1) -U_t(y_2)\big) = 0 \quad \text{for all } x\in \mathcal{O},~ y_1,\,y_2 \in L_t^x \setminus B,
\end{equation}
which implies that $U_t$ belongs to $\H_{S_t}$.
\end{proof}

\begin{proposition}\label{df:eul_velocity}\label{P:TWO}
	There exists a velocity $u_t \in L^2(\R, \rho_t)$ such that
	\begin{equation}
	U_t(y) = u_t\big(X_t(y)\big) 
	\quad \text{for $\bar{\rho}$-a.e. $y\in\R$, where}\quad 
	\rho_t := (X_t)_\# \bar{\rho}.
	\end{equation}
\end{proposition}

\begin{proof}
Since $X_t$ belongs to $\tK$, for all $x\in \R$ there exists at most one $y \in \R \setminus N_{X_t}$ (where $N_{X_t}$ is some null set associated to $X_t$) with $X_t(y) = x$. We can then set
\begin{equation}
u_t(x) := \begin{cases}
~ U_t(y) \quad & \text{if such } y \text{ exists} \\
~ 0      \quad & \text{otherwise.}
\end{cases}
\end{equation}
Then $U_t(y) = u_t(X_t(y))$ for a.e. $y\in \R\setminus N_{X_t}$ and $\|u_t\|_{L^2(\R, \rho_t)} = \|U_t\|_{L^2(\R, \bar{\rho})}$.
\end{proof}

\begin{lemma}
	The space $\H_{S_t}$ defined in \eqref{df:H_S} is characterized as 
	\[\H_{S_t} = \left\{ W \in L^2(\R, \bar{\rho})\colon \text{there exists $w \in L^2(\R, \eta_t)$ with $W=w\circ S_t$} \right\} \]
	where $\eta_t := (S_t)_\# \bar{\rho}$. 
\end{lemma}

\begin{proof}
Any $W \in \H_{S_t}$ is essentially constant on each maximal interval of $\Omega_{S_t}$ (see Remark~\ref{R:EXPLAN} for notation). 
For all $x \in \R \setminus \mathcal{O}_t$ there exists at most one $y \in L^t_x\setminus N_{S_t}$ such that $S_t(y) = x$. Let us therefore define
\[ w(x) := \begin{cases}
~ W(y) \quad & \text{if such } y \text{ exists} \\
~ 0    \quad & \text{otherwise.}
\end{cases} \]
For all $x\in \mathcal{O}_t$, since $W$ is a.e. constant on $L_t^x$, we can pick a generic $y \in L_t^x$ and define
\[ w(x) := W(y). \]            
By doing so, we have constructed $w$ such that
\[ W(y) = w(S_t(y)) \quad \text{for } \text{-a.e. } x\in \R\setminus N_{S_t}. \]
We have then
\[
	\int_\R{ |W(y)|^2 \bar\rho(x) \,dy} 
		= \int_\R{ |w(S_t(y))|^2 \bar\rho(y) \,dy}
		= \int_{\R}{ |w(x)|^2 \,\eta_t(dx)},
\]
with $\eta_t$ as defined above.
\end{proof}

\begin{proposition}\label{prop:incl_Hx}
	The space $\H_{S_t}$ is included in the $L^2(\R,\bar{\rho})$-closure of
	\[ T_{S_t}K \cap \Big[ \Xf_t - X_t \Big]^\perp. \]
\end{proposition}

\begin{proof}
Due to the previous lemma, we are led to show that
\[ \varphi \circ S_t \in  T_{S_t}K \cap \Big[\Xf_t - X_t \Big]^\perp \quad
\text{for all $\varphi \in \D(\R)$,} \]
to get the desired result.
We consider $h>\|\varphi\|_{L^\infty}$ and set
\[ Z_h^\pm = \Big(\id \pm \dfrac{1}{h} \varphi\Big) \circ S_t \in K. \]
We then have
\[ \varphi \circ S_t  = h(Z_h^+ - S_t),  \]
which is by definition an element of  the tangent cone $T_{S_t}K$; see \eqref{df:tangent_cone}.
On the other hand, using the fact that $X_t$ is the projection of $\Xf_t$, we get
\begin{align*}
\pm \langle \, \Xf_t - X_t, \varphi \circ S_t \,\rangle 
& = h \langle \, \Xf_t - X_t, Z_h^\pm - S_t \,\rangle \\
& = h \langle \, \Xf_t - X_t, \tilde{Z}_h^\pm - X_t \,\rangle  \leq 0
\end{align*}
where $\tilde{Z}_h^\pm = Z_h^\pm + \tX \in \tK$, which proves that $\varphi \circ S_t \in \Big[ \Xf_t - X_t \Big]^\perp$.
\end{proof}

\begin{proposition}
The Lagrangian velocity $U_t$ is the orthogonal projection of $\Uf_t$ onto $\H_{S_t}$.	
\end{proposition}

\begin{proof}
We already know that $U_t \in \H_{S_t}$. Let us therefore show that
\begin{equation}
\langle\, \Uf_t -U_t, U_t \, \rangle = 0 \quad \text{and} \quad \langle\, \Uf_t -U_t, W \, \rangle \leq 0 \quad \text{for all}~ W \in \H_{S_t}.
\end{equation}

\medskip

\textbf{Step~1.} We that that $\langle\, \Uf_t -U_t, U_t \, \rangle \geq 0$.

For any $t,h\in \R$, the quantity
	\[ \dfrac{X_{t+h} - X_t}{h}\]
	is uniformly bounded.
	Therefore there exists a sequence $(h_n)$ such that 
	\[U_t^n := \dfrac{X_{t+h_n} - X_t}{h_n} ~ \weak ~ U_t \quad \text{weakly in $L^2(\R, \bar{\rho})$.}\]
	Using the fact that $X_{t+h_n}$ is the projection of $\Xf_{t+h_n}$ onto $\tK$, we have the inequality
	\[ \la \Xf_{t+h_n} - X_{t+h_n} , X_t - X_{t+h_n} \ra \leq 0, \]
	which can also be rewritten as
	\[ \la \Xf_t - X_t -h_nU_t^n + \int_t^{t+h_n}{\Uf_s \,ds}, -h_n U_t^n \ra \leq 0. \]
Equivalently, by splitting the powers of $h_n$, we have
	\[ -h_n \la \Xf_t - X_t, U_t^n \ra - h_n^2 \la U_t^n - \dfrac{1}{h_n} \int_t^{t+h_n}{\Uf_s \,ds}, U_t^n \ra \leq 0. \]
	Since $h_n U_t^n = X_{t+h_n} - X_t$ and since $X_t$ is the projection on $\tK$ of $\Xf_t$, we deduce that the first term of the left-hand side is non-negative and thus 
	\[ - h_n^2 \la U_t^n - \dfrac{1}{h_n} \int_t^{t+h_n}{\Uf_s \,ds}, U_t^n \ra \leq 0.\]  
	As $h_n \rightarrow 0^+$ we have then
	\[ \dfrac{1}{h_n} \int_t^{t+h_n}{\Uf_s \,ds} \longrightarrow \Uf_t \quad \text{strongly in}~ L^2(\R, \bar{\rho}). \]
From the weak convergence of $U_t^n$ towards $U_t$, it follows that 
\begin{gather*}
	\la \dfrac{1}{h_n} \int_t^{t+h_n}{\Uf_s \,ds}, U_t^n \ra \longrightarrow \la \Uf_t , U_t \ra,
\\
	\|U_t\|_{L^2(\R,\bar{\rho})} \leq \liminf \|U_t^n\|_{L^2(\R,\bar{\rho})}.
\end{gather*}
So we finally obtain the desired inequality 
\[
	\langle\, \Uf_t -U_t, U_t \, \rangle \geq 0.
\]

\medskip

\textbf{Step~2.} We show that $\langle\, \Uf_t -U_t, W \, \rangle \leq 0$ for all $W\in \H_{S_t}$:

Thanks to Propositions~\ref{prop:space_velocity} and~\ref{prop:incl_Hx}, there exists $h>0$ and $Z_t \in K$ such that
	\[ W = h(Z_t - S_t) \quad \text{and} \quad \langle\,\Xf_t-X_t , Z_t -S_t \, \rangle = 0.\]
We must show that
	\[ \langle\, \Uf_t - U_t, Z_t-S_t \, \rangle \leq 0. \]
	We consider as before the approximate velocity $U_t^n$ and introduce $\delta_n := U_t^n-U_t$.
	Since $X_{t+h_n}$ is the projection of $\Xf_{t+h_n}$ onto $\tK$, we have
\begin{align*}
	0 & \geq \langle\, \Xf_{t+h_n} - X_{t+h_n}, Z_t + \tX - X_{t+h_n} \, \rangle \\
	& = \langle\, \Xf_t -X_t + h_n\left(\dfrac{1}{h_n}\int_t^{t+h_n}{\Uf_s \,ds} - U_t\right) - h_n\delta_n, (Z_t-S_t) - h_n U_t -h_n\delta_n \, \rangle.
\end{align*}
Rearranging the terms, we can then get
	\begin{align*}
	& h_n \langle\, \dfrac{1}{h_n}\int_t^{t+h_n}{\Uf_s \,ds} - U_t, Z_t-S_t \, \rangle \\
	& \qquad \leq - \langle\, \Xf_t - X_t, Z_t - S_t \, \rangle  + h_n \la \Xf_t - X_t,U_t \ra \\
	& \qquad\quad + h_n \Big(\la \Xf_t - X_t,\delta_n \ra + \langle\, \delta_n, Z_t - S_t \, \rangle  \Big) \\
	& \qquad\quad + h_n^2 \Big( \langle\,\dfrac{1}{h_n}\int_t^{t+h_n}{\Uf_s \,ds} - U_t, U_t \, \rangle + \langle\, \dfrac{1}{h_n}\int_t^{t+h_n}{\Uf_s \,ds} -U_t, \delta_n \, \rangle \Big) \\
	& \qquad\quad + h_n^2 \Big(\langle\, \delta_n, U_t \, \rangle -\|\delta_n\|^2_{L^2(\R,\bar{\rho})} \Big).         
\end{align*}
By definition of $Z_t$ and Proposition~\ref{prop:incl_Hx}, the first line of the right-hand side vanishes.
	Dividing now by $h_n$ and letting $h_n \rightarrow 0$, the remaining terms tend to $0$ and we get
	\begin{equation}\label{eq:ineg_polarcone}
	\langle\, \Uf_t - U_t, Z_t-S_t \, \rangle \leq 0.
	\end{equation}
It follows that $U_t$ is the orthogonal projection of the free velocity $\Uf_t$ onto $\H_{S_t}$.
\end{proof}

\begin{remark}\label{R:KKPO}
Since $U_t$ is the orthogonal projection of $\Uf_t$ onto $\H_{S_t}$, we have
\begin{equation}
	U_t(y) = \begin{cases}
		~ \Uf_t(y) & \text{if $y \in \R\setminus\Omega_t$}
\\
		~ \displaystyle \frac{1}{\bar{\rho}(I)} \int_I \Uf_t(z) \,\bar{\rho}(dz)
			& \text{if $y \in I$ with $I \in \mathcal{J}(\Omega_t)$,}
	\end{cases}
\label{eq:expr_Ut}
\end{equation}
where $\Omega_t := \bigcup_{z\in\mathcal{O}_t} L^z_t$  and $\mathcal{J}(\Omega_t)$ denotes the set of maximal intervals contained in $\Omega_t$ (there are at most countably many).
We refer the reader to Remark~\ref{R:EXPLAN} for notation.
\end{remark}

\paragraph{Recovering of the mass equation.}
As explained in Section~\ref{sec:mainresult}, the density
\begin{equation}
\rho_t := (X_t)_\# \bar{\rho}, \quad X_t \in \tK,  
\end{equation}
is absolutely continuous with respect to the Lebesgue measure and satisfies the  constraint 
\begin{equation}
0 \leq \rho_t \leq 1 \quad \text{a.e.}
\end{equation}
For all $\xi \in \C^\infty_c([0,T) \times \R)$ we have by a change a variable
\begin{align*}
	- \int_\R{\xi(0,y) \,\bar{\rho}(y) \,dy}
		& = \int_0^T{\dfrac{d}{dt} \left(\int_\R{\xi(t, X_t(y)) \,\bar{\rho}(y) \,dy}\right) \,dt}\\
		& = \int_0^T{\int_\R{\left(\partial_t \xi(t, X_t(y)) + \frac{d}{dt}X_t(y) \partial_x\xi(t, X_t(y))\right)\bar{\rho}(y) \,dy} \,dt}\\
		& = \int_0^T{\int_\R{\Big(\partial_t \xi(t, X_t(y)) + U_t(y) \partial_x\xi(t, X_t(y))\Big) \bar{\rho}(y) \,dy} \,dt},  
\end{align*}
which gives the weak formulation of the mass equation~\eqref{eq:pressurelessMOD_mass} (see Proposition~\ref{P:TWO}):
\begin{equation} 
- \int_\R{\xi(0,x) \,\rho_0(x) \,dx} = \int_0^T{\int_\R{\left(\partial_t \xi(t, x) + u_t(x) \partial_x\xi(t, x)\right) \rho_t(x) \,dx} \,dt}.
\end{equation}


\subsection{Memory effects, definition of the adhesion potential}{\label{sec:Gamma}}

By analogy with the discrete model~\eqref{eq:glueyparticle}, we define the adhesion potential 
\begin{equation}\label{eq:df_gamma}
\Gamma_t(y) := \int_{-\infty}^y{\big(U_t(z) - \Uf_t(z) \big) \,\bar{\rho}(dz)}
	\quad\text{for a.e. $t\in (0,T)$, $y\in\R$.}
\end{equation}

\begin{proposition}\label{P:SUPPORT}
	For a.e. $t\in (0,T)$, $y\in \R$, we have $\Gamma_t(y) \leq 0$ and $\spt \, \Gamma_t \subset \Omega_t$.
\end{proposition}

\begin{proof}
We use the notation introduced in Remark~\ref{R:KKPO}. For any $y\in\R$ let
\[
	\mathcal{J}_y := \big\{ I \in \mathcal{J}(\Omega_t) \colon
		\sup I \leq y \big\}.
\]
Fix some $y \in \R\setminus\Omega_t$. 
We decompose the integral defining $\Gamma_t(y)$ and use~\eqref{eq:expr_Ut} to write
\begin{align*} 
	\Gamma_t(y) & = \int_{(-\infty,y]\setminus\Omega_t}{\Big(U_t(z) - \Uf_t(z)\Big)  \,\bar{\rho}(dz)}
		+ \sum_{I \in \mathcal{J}_y} \int_I{\Big(U_t(z) - \Uf_t(z)\Big) \,\bar{\rho}(dz)} \\
& = 0 + \sum_{I \in \mathcal{J}_y} \left[\int_I{\left(\dfrac{1}{\bar{\rho}(I)} \int_I{\Uf_t(\tilde{z})\bar{\rho}(d\tilde{z})}\right) \,\bar{\rho}(dz)} - \int_I{\Uf_t(z) \,\bar{\rho}(dz)} \right] = 0.
\end{align*}
An integration by parts (with $\Gamma_t$ continuous and $\partial_y S_t$ a measure) now yields
\begin{equation}
\langle\, \partial_y \Gamma_t , S_t \, \rangle = 0
\end{equation}
(see also \cite{natile2009} Lemma~3.10). Using \eqref{eq:df_gamma} and
\eqref{eq:ineg_polarcone}, we obtain
\begin{align*} 
\langle\, -\partial_y \Gamma_t , Z -S_t  \,\rangle
& = \langle\, \Uf_t - U_t , Z -S_t  \,\rangle \leq 0 \\
& \qquad \text{for all}~Z~\text{in}~ K~ \text{with}~ \langle \Xf_t - X_t, Z-S_t \rangle = 0 . 
\end{align*} 
Suppose that in addition $Z\in \C^1(\R)$. Then
\[ 
	0 \geq \langle\, -\partial_y \Gamma_t , Z - S_t  \,\rangle 
		= \la -\partial_y \Gamma_t , Z \ra  
		= \int_\R{\Gamma_t(y) \partial_y Z(y) \,\bar{\rho}(dy)}.
\]
From the arbitrariness of the test function $Z$, we obtain that
$\Gamma_t \leq 0$.
\end{proof}

As in Proposition~\ref{df:eul_velocity}, we can define for
a.e. $t\in (0,T)$ an Eulerian adhesion potential
\begin{equation}
\gamma_t(x) := \Gamma_t(y) \quad \text{with} \quad x=X_t(y).
\label{E:TRGAMMA}
\end{equation}

\paragraph{Exclusion relation.}
If $\rho_t$ is a Borel family of probability measures satisfying the continuity equation in the distributional sense for a Borel velocity field $u_t$ such that
\[ \int_0^T\int_{\R}{|u_t| \,\rho_t(dx) \,dt} < + \infty, \]
then there exists a narrowly continuous  curve $t\in [0,T] \mapsto \widetilde{\rho}_t \in \mathcal{P}(\R)$ such that
\[ \rho_t = \widetilde{\rho}_t \quad \text{for a.e.}~t\in (0,T);\]
see \cites{ambrosio2008, santambrogio2015}, for instance.
Recall now that the transport $X_t$ satisfies a Lipschitz property, which has allowed us to define the velocity $U_t$. 
We have that
\begin{align}
	X \quad & \text{belongs to} \quad W^{1,\infty}(0,T;L^2(\R,\bar{\rho})) \\
	U, \Uf \quad & \text{belongs to} \quad L^\infty(0,T;L^2(\R,\bar{\rho})), 
\end{align}
which implies that $\partial_y \Gamma$ is in $L^\infty(0,T;L^1(\R))$ since $\bar{\rho}$ is absolutely continuous with respect to the Lebesgue measure. The adhesion potential $\Gamma_t$ is then bounded and continuous in space, hence $\gamma_t$ is measurable and bounded. It can be paired with $\rho_t$, which is absolutely continuous with respect to the Lebesgue measure (with the pointwise bound $0\leq \rho_t\leq 1$). We have
\begin{align*}
	\int_\R \gamma_t(x) \rho_t(x) \,dx 
		& = \int_\R \gamma_t(x) \,\big( (X_t)_\#\bar{\rho} \big)(dx)
\\
		& = \int_\R \gamma_t(X_t(y)) \bar{\rho}(y) \,dy
		= \int_\R \Gamma_t(y) \bar{\rho}(y) \,dy.
\end{align*}
On the other hand, recall that $\bar{\rho}(y)/\partial_y \tX(y) = 1$ for $\bar{\rho}$-a.e. $y\in\R$, 
because of \eqref{E:EUQ}. Since $\Gamma_t$ vanishes outside $\Omega_t$,
as shown in Proposition~\ref{P:SUPPORT}, we can write
\begin{align*}
	\int_\R \Gamma_t(y) \bar{\rho}(y) \,dy
		& = \int_\R \Gamma_t(y) \bar{\rho}(y) \,dy
		= \int_{\Omega_t} \Gamma_t(y) \partial_y \tX(y) \,dy
\\
		& = \int_{\Omega_t} \Gamma_t(y) \partial_y X_t(y) \,dy
		= \int_\R \Gamma_t(y) \partial_y X_t(y) \,dy
		= \int_\R \gamma_t(x) \,dx.
\end{align*}
For the third equality, we have used that 
$\partial_yX_t(y) = \partial_y\tX(y)$ for a.e. $y\in\Omega_t$.
Therefore
\begin{equation*}
	\int_\R \gamma_t(x)\rho_t(x) \,dx = \int_\R \gamma_t(x) \,dx
\end{equation*}
for a.e. $t\in(0,T)$. 
Since in addition $(1-\rho_t(x))\gamma_t(x) \leq 0$ for a.e. $x\in\R$, we get
\begin{equation}\label{eq:excl_constraint_0}
(1-\rho) \gamma = 0 \quad \text{a.e. on } (0,T) \times \R.
\end{equation}

\subsection{Recovering of the momentum equation}{\label{sec:weak_form}}

Similarly to the continuity equation, we want to recover the Eulerian momentum equation~\eqref{eq:pressurelessMOD_mom} by passing to the Lagrangian coordinates.
For all $\varphi \in \mathcal{C}^\infty_c\big([0, T) \times \R\big)$ we have
\begin{align}
	- \int_\R \varphi(0,x) \, \bar{\rho}(x) \bar{u}(x) \,dx 
	& = -\int_\R \varphi(0,X_0(y)) \,\bar{\rho}(y)\bar{U}(y) \, dy 
\nonumber\\
	& = \int_0^T \int_\R \frac{d}{dt}\varphi(t,X_t(y)) 
		\, \bar{\rho}(y)\bar{U}(y) \,dy \,dt.
\label{E:START}
\end{align}

\noindent Recall our choice of initial data \eqref{eq:initiald} (in particular, we have $X_0 = \id$), 
from which the first equality follows.
We can expand the time derivative of the test function to obtain
\[
	\frac{d}{dt}\varphi(t,X_t(y))
		= \partial_t\varphi(t,X_t(y)) + U_t(y) \partial_x \varphi(t,X_t(y)),
\]
where we used that $\frac{d}{dt} X_t = U_t$, by definition. 
Using \eqref{eq:df_gamma} and \eqref{eq:freevelocity}, we find
\[
	\partial_y \Gamma_t(y) 
		=  \bar{\rho}(y)\bigg( U_t(y)-\bar{U}(y)-\int_0^t f(s,X_s(y)) \,ds \bigg)
\]
for a.e. $y\in\R$. Rearranging terms, we obtain from this the identity
\[
	\bar{\rho}(y) \bar{U}(y)
		=  \bar{\rho}(y)\bigg( U_t(y) - \int_0^t f(s,X_s(y)) \bigg) - \partial_y \Gamma_t(y),
\]
which we insert into \eqref{E:START}. Let us discuss the different terms. First, we have
\begin{align*}
	& \int_0^T \int_\R \frac{d}{dt}\varphi(t,X_t(y)) \,\bar{\rho}(y) U_t(y) \,dy \,dt
\\
	& \qquad
		= \int_0^T \int_\R \Big( \partial_t\varphi(t,X_t(y)) + U_t(y) \partial_x \varphi(t,X_t(y))
			\Big) \,\bar{\rho}(y) U_t(y) \,dy \,dt
\\
	& \qquad
		= \int_0^T \int_\R \Big( \partial_t\varphi(t,x) + u_t(x) \partial_x \varphi(t,x)
			\Big) \,\rho_t(x)u_t(x) \,dx \,dt,
\end{align*}
where we used Proposition~\ref{P:TWO}. Second, by integrating by parts in time, we get
\begin{align*}
	& \int_0^T \int_\R \frac{d}{dt}\varphi(t,X_t(y)) \,\bar{\rho}(y) 
		\bigg( -\int_0^t f(s,X_s(y)) \,ds \bigg) \,dy \,dt
\\
	& \qquad
		= \int_0^T \int_\R \varphi(t,X_t(y)) \,\bar{\rho}(y) f(t,X_t(y)) \,dy \,dt
\\
	& \qquad
		= \int_0^T \int_\R \varphi(t,x) \,\rho_t(x) f(t,x) \,dx \,dt.
\end{align*}
Finally, from definition \eqref{E:TRGAMMA} and the chain rule, we obtain the identity
\[
	\partial_y \Gamma_t(y) 
		= \partial_y \Big( \gamma_t(X_t(y)) \Big)
		= \partial_x \gamma_t(X_t(y)) \; \partial_y X_t(y)
\]
for a.e. $y\in\R$. It then follows that
\begin{align*}
	& -\int_0^T \int_\R \frac{d}{dt}\varphi(t,X_t(y)) \,\partial_y \Gamma_t(y) \,dy \,dt
\\
	& \qquad
		= -\int_0^T \int_\R \Big( \partial_t\varphi(t,X_t(y)) + U_t(y) \partial_x \varphi(t,X_t(y))
			\Big) \,\partial_x \gamma_t(X_t(y)) \; \partial_y X_t(y) \,dy \,dt
\\
	& \qquad
		= -\int_0^T \int_\R \Big( \partial_t\varphi(t,x) + u_t(x) \partial_x \varphi(t,x)
			\Big) \,\partial_x \gamma_t(x) \,dx \,dt.
\end{align*}
Combining all terms, we find the momentum equation \eqref{eq:pressurelessMOD_mom}, 
which concludes the proof.

\begin{remark}
We have uniqueness for \eqref{eq:pressurelessMOD} in the class of weak solutions of the form
\[ \rho_t = X_t \# \bar{\rho}, \quad X_t = \P_{\tK}(\Xf_t).\]
Indeed, by the contraction property of the metric projection, for solutions $X^1_t, X^2_t$ we have
\begin{align*}
\|X^1_t - X^2_t\|_{L^2(\R,\bar{\rho})} 
& = \left\|\P_{\tK}\left(\id + \int_0^t{U^\text{free,1}_s \,ds} \right)-\P_{\tK}\left(\id + \int_0^t{U^\text{free,2}_s \,ds}\right)\right\|_{L^2(\R,\bar{\rho})}  \\
& \leq \left\|\int_0^t{U^\text{free,1}_s \,ds}-\int_0^t{U^\text{free,2}_s \,ds}\right\|_{L^2(\R,\bar{\rho})} \\
& \leq t \|\bar{U}^1 - \bar{U}^2\|_{L^2(\R,\bar{\rho})} + k\int_0^t\int_0^s{\|X_\tau^1 - X_\tau^2\|_{L^2(\R,\bar{\rho})} \,d\tau \,ds}, 
\end{align*}
with $k$ the Lipschitz constant of $f$.
From Gronwall's lemma (see \cite{bainov2013} Theorem 11.4), we get
\begin{equation}
\|X^1_t - X^2_t\|_{L^2(\R,\bar{\rho})}  \leq t \|\bar{U}^1 - \bar{U}^2\|_{L^2(\R,\bar{\rho})} 
	+ k\int_0^t{\frac{s^2}{2} \exp\Big( k(t-s) \Big) \|\bar{U}^1 - \bar{U}^2\|_{L^2(\R,\bar{\rho})} \,ds},
\end{equation}
which proves that $X_t^1=X_t^2$ for all $t$ if $\bar{U}^1= \bar{U}^2$, 
and thus the uniqueness of the transport $X_t$. 
The velocity $U_t$ is then uniquely defined as well since it is the orthogonal projection of 
\[
	\Uf_t = \bar{U} + \int_0^t f(s,X_s) \,ds
\]
onto $\mathcal{H}_{S_t}$ for a.e. $t$.
Finally, the adhesion potential $\Gamma_t$ is unique by definition~\eqref{eq:df_gamma}.
\end{remark}

\begin{remark}
The initial data is actually attained in a stronger sense than just distributionally
(cf. Definition~\ref{D:SOLNN}) as $t\rightarrow 0$. 
Let us define the $L^2$-Wasserstein distance
\[ 
	W_2(\rho^1, \rho^2)^2 := \min \left\{ 
		\int_{\R \times \R}{|x_1-x_2|^2 \,\omega(dx_1,dx_2)} \colon
			\omega \in \mathcal{P}(\R\times \R), \ \pi^i_\# \omega = \rho_i \right\} 
\]
where $\pi^i(x_1,x_2) = x_i$ is the projection on the $i$th coordinate.
In the one-dimensional setting, there exists a unique optimal coupling $\omega$:
Denoting by $X_i$ the monotone transport in $L^2(\R,\bar{\rho})$ 
such that $\rho_i = X_{i\#}\bar{\rho}$, 
where $\bar{\rho}$ is some reference measure 
that is absolutely continuous with respect to the Lebesgue measure, 
we can write 
\[ 
	W_2(\rho^1,\rho^2)^2 
		= \int_{\R} \left| X_1(y) -X_2(y) \right|^2 \,\bar{\rho}(dy);
\] 
see \cite{rachev1998}).
If we introduce additionally the semi-distance
\[ 
	U_2\big((\rho^1,\rho^1 u^1), (\rho^2,\rho^2 u^2)\big)^2 
		:= \int_{\R}{\big|u^1(X_1(y)) - u^2(X_2(y))\big|^2 \,\bar{\rho}(dy) }, 
\]
then the function
\[
	D_2\big((\rho^1,\rho^1 u^1), (\rho^2,\rho^2 u^2)\big) 
		:= W_2(\rho^1, \rho^2) + U_2\big((\rho^1,\rho^1 u^1), (\rho^2,\rho^2 u^2)\big)
\]
is a distance on the space 
\[
	\mathcal{V}_2(\R) := \Big\{ (\rho, \rho u) \in \mathcal{P}_2(\R) \times \mathcal{M}(\R)
		\colon u \in L^2(\R, \rho)\Big\}.
\]
One can show that $(\mathcal{V}_2(\R), D_2)$ is a metric space, not necessarily complete.
Convergence with respect to the distance is stronger than weak convergence of measures.
We refer the reader to 
see~\cite{natile2009} Proposition 2.1 and~\cite{ambrosio2008} Definition 5.4.3
for further information.

The density $\rho_t$ converges to $\rho_0$ for the Wasserstein distance since
\begin{align*}
	W_2(\rho_t,\rho_0)^2
		& = \int_{\R}{|X_t(y) -X_0(y)\big|^2 \,\bar{\rho}(dy)} \\
		& = \| X_t-\id \|_{L^2(\R,\bar{\rho})}^2 \\
		& \leq 2t \|\bar{U}\|_{L^2(\R,\bar{\rho})}^2
			+ 2\left\|\int_0^t{ \int_0^s{f(\tau,X_\tau) \,d\tau} \,ds}\right\|_{L^2(\R,\bar{\rho})}^2
	\longrightarrow 0
\end{align*}
as $t\rightarrow 0$. 
Moreover, we can adapt the proof of \cite{brenier2013} Theorem~3.5 to show that 
\begin{equation}
	U_t \longrightarrow \bar{U}
	\quad\text{strongly in $L^2(\R,\bar{\rho})$}
\label{E:STRONG}
\end{equation}
\emph{provided} that the initial velocity $\bar{U}$ belongs to the tangent
cone $\mathbb{T}_{S_0} K$ with $S_0 := X_0-\tX \in K$, or even to
$\mathcal{H}_{S_0}$; see \eqref{df:tangent_cone} and \eqref{df:H_S}.
As follows from the proof of Proposition~\ref{prop:space_velocity},
this requires that the initial velocity is non-decreasing (resp.
constant) on the congested zones of the initial density. If this
condition is not satisfied, then the initial velocity $\bar{U}$ may not be
attained, not even in distributional sense; see Remark~\ref{E:INVELO}.
Convergence \eqref{E:STRONG} translates into
\[
	D_2\big((\rho_t,\rho_t u_t), (\bar\rho,\bar\rho \bar{u})\big) \longrightarrow 0,
\]
from which it follows that $(\rho_t, \rho_t u_t) \longrightarrow
(\bar\rho, \bar\rho \bar{u})$ in $\mathcal{V}_2(\R)$ as $t\rightarrow 0$.
\end{remark}

\section{Numerical simulation}\label{sec:numerics}

\begin{figure}[t]
	\begin{center} \includegraphics[width = 10cm]{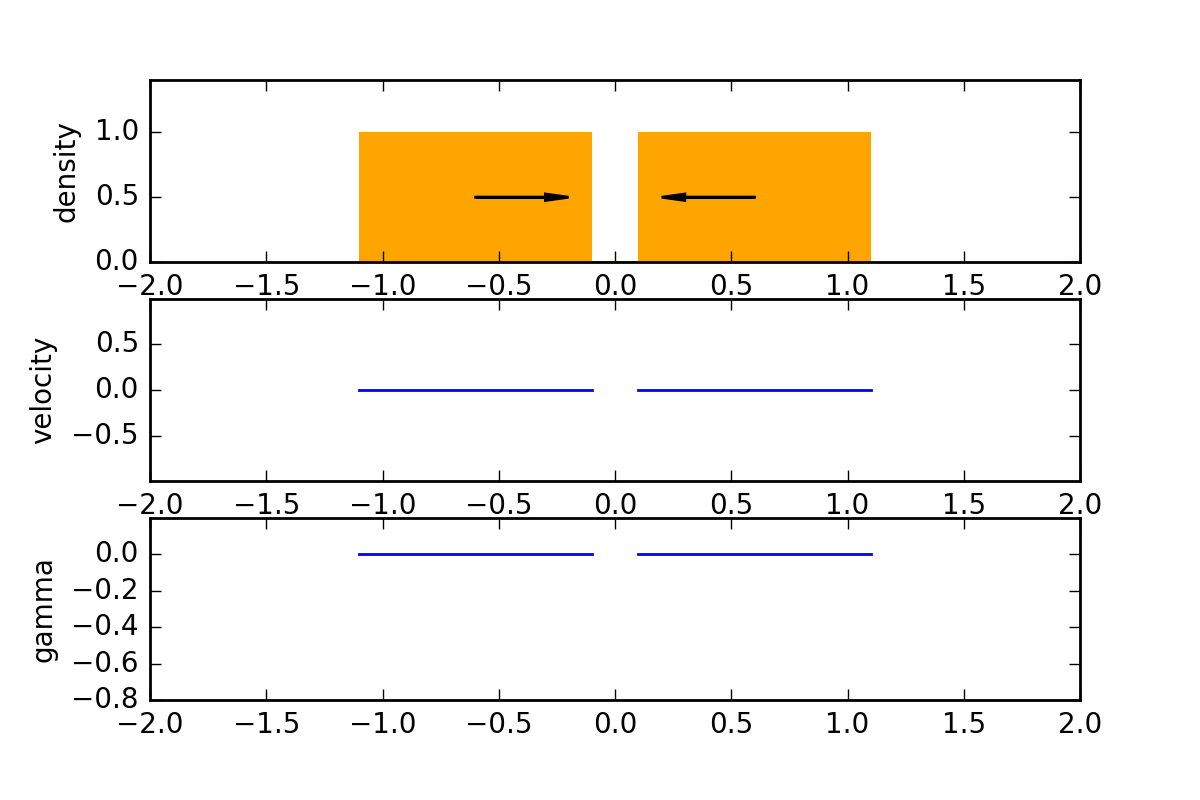}	
	\end{center}
	\caption{{\label{fig:init_data}} Initial data. The arrows represent the external force $f$ at time $0$.}
\end{figure}

To illustrate the memory effects of one-dimensional flows in \eqref{eq:pressurelessMOD}, 
we consider initial data formed by two congested blocks $\mathbf{1}_{[a_1,b_1]}$ 
and $\mathbf{1}_{[a_2,b_2]}$ with $b_1 < a_2$ at time $t=0$; see Figure~\ref{fig:init_data}.
Initially, both the velocity and adhesion potential are equal to zero.
We apply an external force $f$, such that the system first compresses and then decompresses in a second phase:
\begin{equation}
f(t,x) = 
\begin{cases}
~ \alpha \quad & \text{if} \quad x < 0 \\
~ -\alpha\quad & \text{if} \quad x \geq 0
\end{cases}
\qquad \text{for}\quad t \leq t^*
\end{equation}
\begin{equation}
f(t,x) = 
\begin{cases}
~ -\alpha \quad & \text{if} \quad x < 0 \\
~ \alpha \quad & \text{if} \quad x \geq 0
\end{cases}
\qquad \text{for}\quad t > t^*.
\end{equation}
In our simulation, we choose $\alpha = 0.5$ and $t^* = 1$.
We denote by $\Xe$, $\Ue$ the exact solution of the dynamics.
The process can be decomposed into four different phases:

\medskip

\textbf{Phase 1.} The blocks move freely until time $t_1= \sqrt{(a_2 -b_1)/\alpha}$. 
Then the blocks collide. 
We choose initial positions in such a way that the collision happens at $x=0$:
\[
\left\{
\begin{aligned}
	~\Xe_t(y) 
		& = X_0(y) + \dfrac{\alpha t^2}{2} \Big(\mathbf{1}_{\{X_0(y) < 0 \}}(y) 
			-\mathbf{1}_{\{X_0(y) > 0 \}}(y) \Big)
\\
	~\Ue_t(y) 
		& = \alpha t \Big(\mathbf{1}_{\{X_0(y) < 0 \}}(y) 
			-  \mathbf{1}_{\{X_0(y) > 0 \}}(y) \Big)
\\
	~\Ge_t(y) 
		& = 0
\end{aligned}
\right.
\qquad\text{for $t \leq t_1$.}
\]

\begin{figure}[t]
	\begin{center}
		\includegraphics[width = 7.5cm]{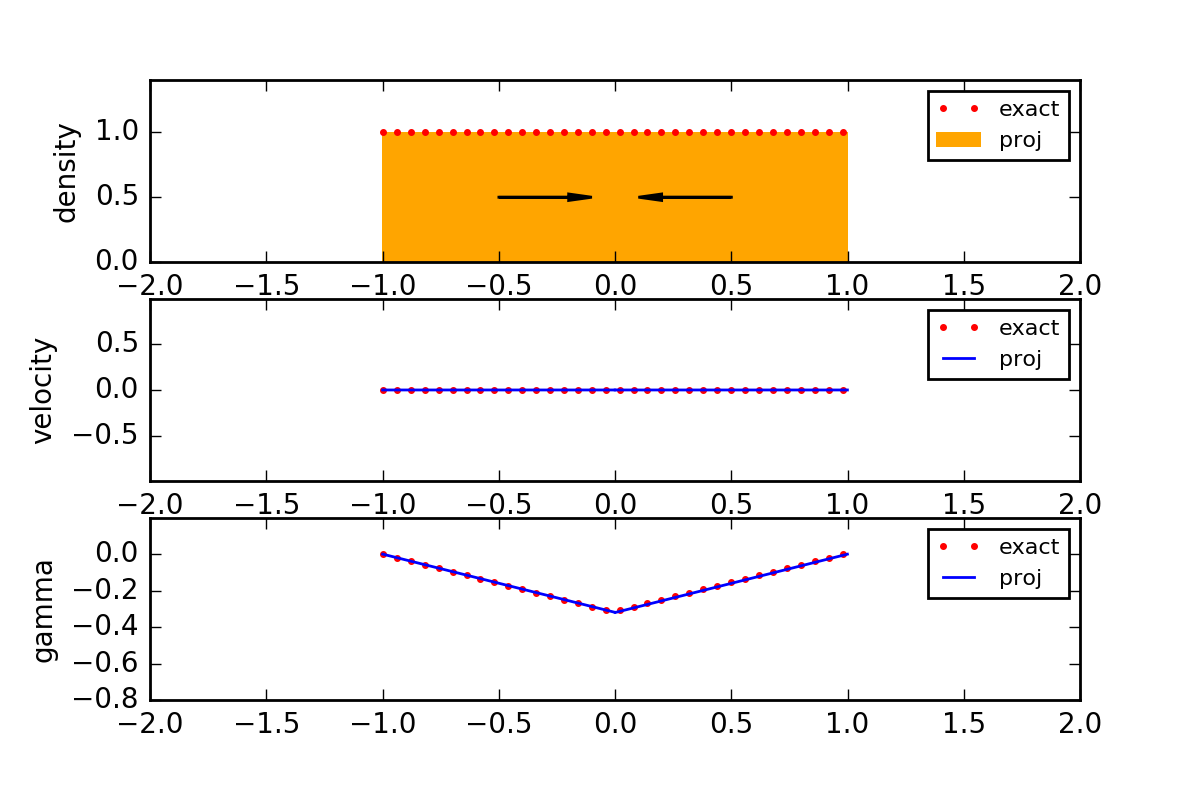} 
		\includegraphics[width = 7.5cm]{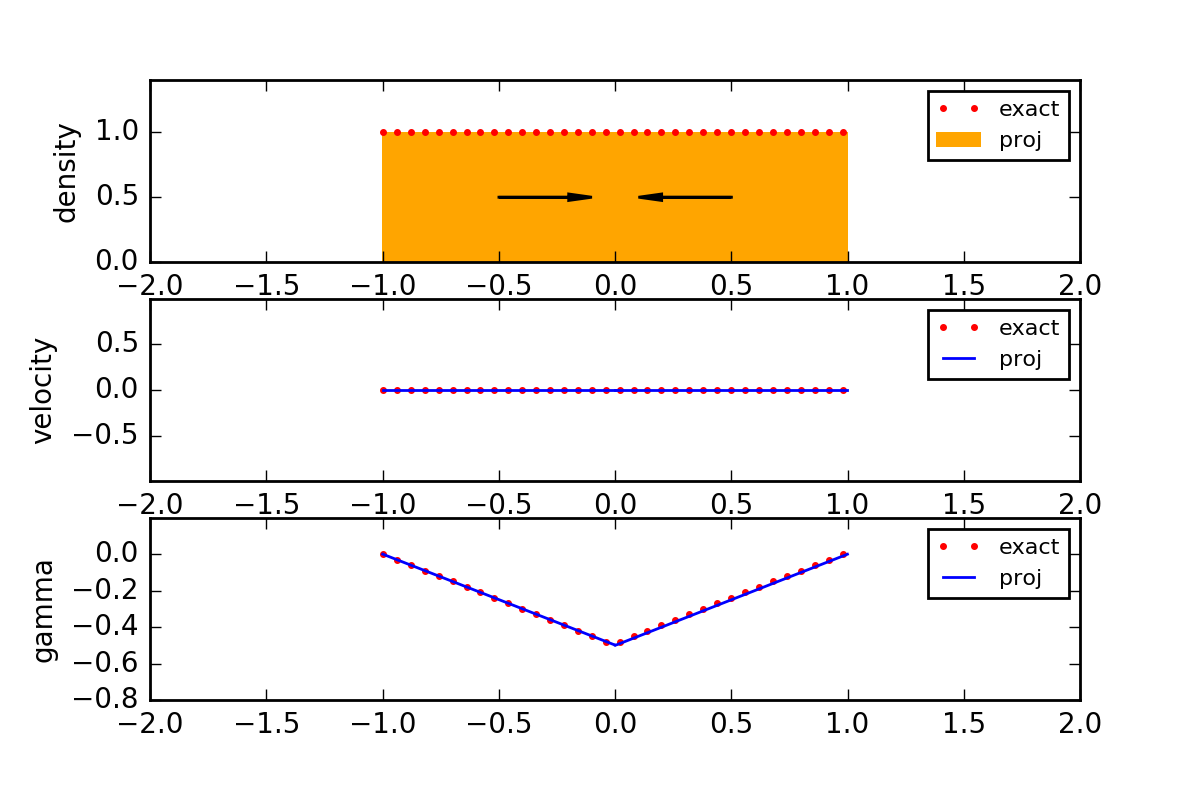}
	\end{center}
	\caption{{\label{fig:congestion}} Solution at times $t =t_1 = 0.64$ (left) 
	and $t=t^*=1$ (right).}
\end{figure}

\textbf{Phase 2.} 
From time $t_1$ on, the blocks are stuck together, 
but the force keeps compressing until time $t=t^*$.
The velocity is $0$ while the adhesion potential is activated:
\[
\left\{
\begin{aligned}
	~\Xe_t(y) 
		& = \Xe_{t_1}(y) 
\\
	~\Ue_t(y) 
		& = 0
\\
	~\Ge_t(y) 
		& = \alpha\big(\Xe_t(y)-(b_2-a_2)\big)t \mathbf{1}_{\{ X_0(y) > 0 \}}(y)
\\
		& -\alpha\big(\Xe_t(y)+(b_1-a_1)\big)t \mathbf{1}_{\{ X_0(y) < 0 \}}(y) 
\end{aligned}
\right.
\qquad\text{for $t_1 < t \leq t^*$.}
\]

\begin{figure}[t]
	\begin{center}
		\includegraphics[width = 7.5cm]{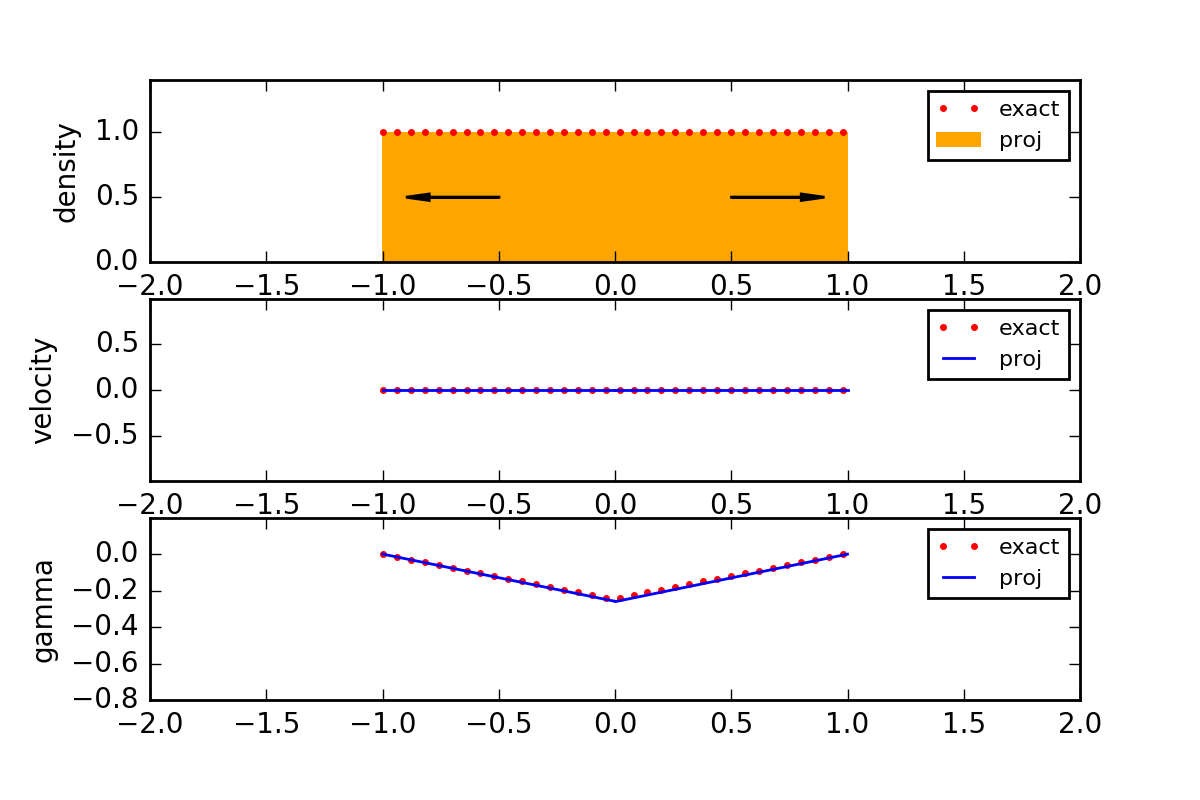} 
		\includegraphics[width = 7.5cm]{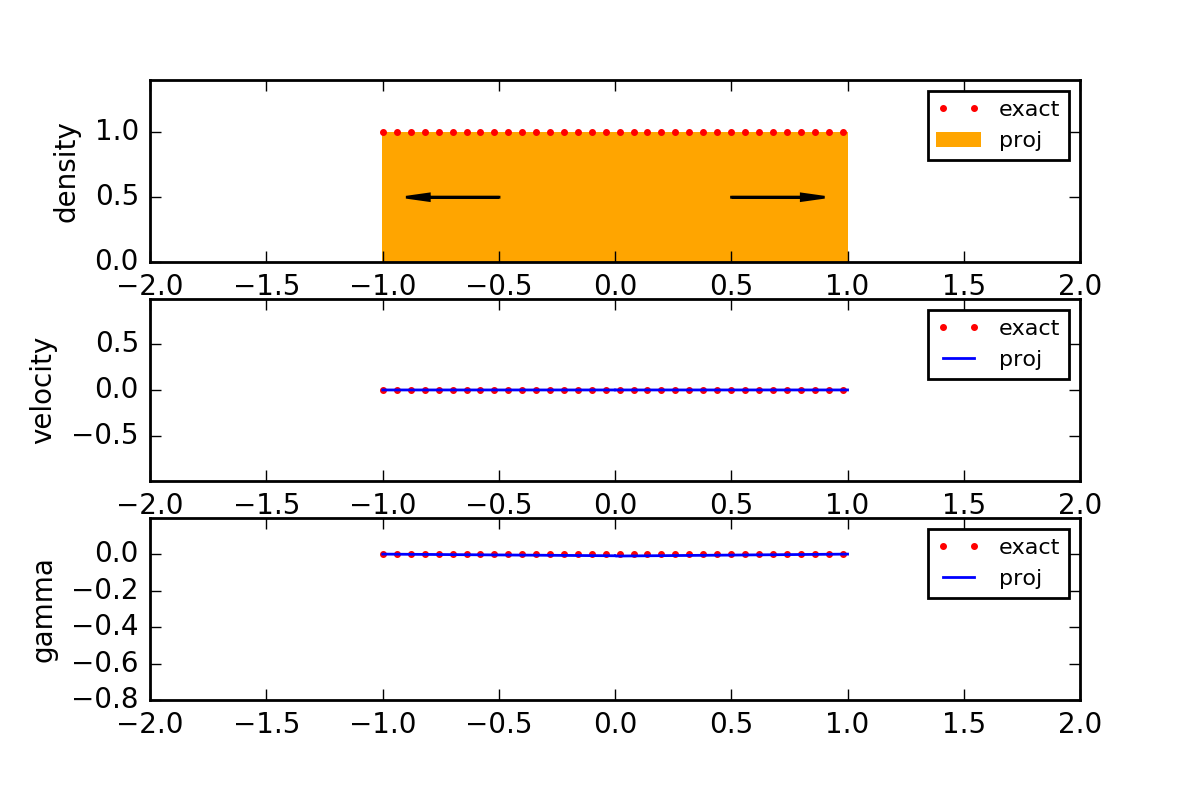}
	\end{center}
	\caption{{\label{fig:decongestion1}} Solution at times $t = 1.5$ (left) 
	and $t=t_2=2$ (right).}
\end{figure}

\textbf{Phase 3.} 
When we reverse the force at time $t=t^*$, the blocks remain stuck to each other 
until the adhesion potential comes back to $0$. The velocity is zero:
\[
\left\{
\begin{aligned}
	~\Xe_t(y) 
		& = \Xe_{t_1}(y) 
\\
	~\Ue_t(y) 
		& = 0 
\\
	~\Ge_t(y)
	& = \Big[-\alpha\big(\Xe_t(y)-(b_1-a_1)\big) t^* 
\\
	& \qquad + \alpha\big(\Xe_t(y)-(b_1-a_1)\big) (t-t^*) \Big]\mathbf{1}_{\{ X_0(y) < 0 \}}(y) 
\\
 	& + \Big[\alpha\big(\Xe_t(y)-(b_2-a_2)\big) t^* 
\\
	& \qquad - \alpha\big(\Xe_t(y)-(b_2-a_2)\big) (t-t^*) \Big]\mathbf{1}_{\{ X_0(y) > 0 \}}(y)
\end{aligned}
\right.
\qquad\text{for $t^* < t \leq t_2$.}
\]

\begin{figure}[t]
	\begin{center}
		\includegraphics[width = 7.5cm]{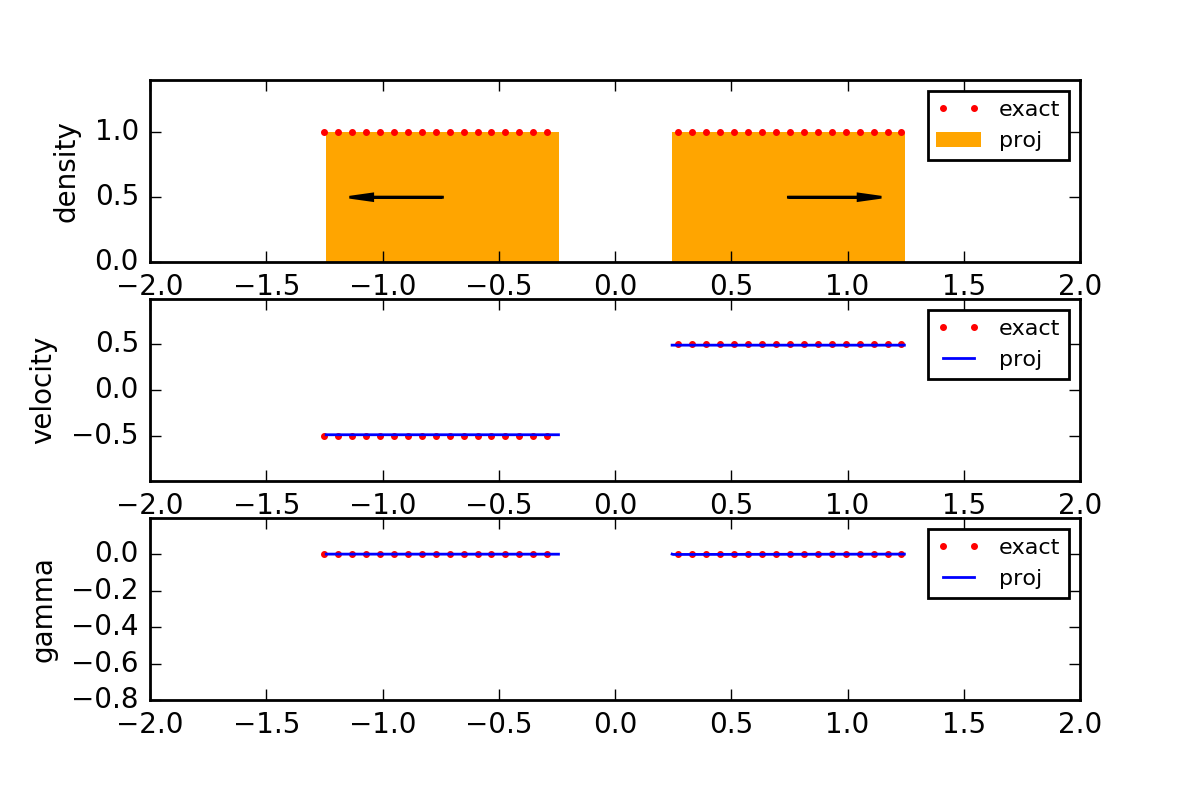}
	\end{center}
	\caption{{\label{fig:decongestion2}} Solution at time $t=3$.}
\end{figure}

\textbf{Phase 4.} Finally, at time $t_2 = 2t^*$ the blocks separate from each other:
\[
\left\{
\begin{aligned}
	~\Xe_t(y) 
		& = \Xe_{t_1}(y) + \dfrac{\alpha (t-t_2)^2}{2} \Big(-\mathbf{1}_{\{X_0(y) < 0 \}}(y) 
			+ \mathbf{1}_{\{X_0(y) > 0 \}}(y) \Big)
\\
	~\Ue_t(y) 
		& = \alpha (t-t_2) \Big(-\mathbf{1}_{\{X_0(y) < 0 \}}(y) 
			+ \mathbf{1}_{\{X_0(y) > 0 \}}(y) \Big)
\\
	~\Ge_t(y) 
		& = 0
\end{aligned}
\right.
\qquad\text{for $t \geq t_2$.}
\]

\medskip
Our numerical code follows the Lagrangian approach developed in the previous sections. 
To determine the transport $X_{t_h}$ at time $t_h$, we minimize the objective function
\[ 
	\phi_{t_h}(X) = \left\| X_0 + \int_0^{t_h}{\Uf_s \,ds} - X \right\|_{L^2(\R,\rho_0}^2 
\] 
under the constraint $X \in \tK$.
This step is performed by use of the Python software CVXOPT for convex optimization; see \texttt{http://cvxopt.org}. 
We discretize in space by considering the congested blocks consisting of two sets
of equally spaced particles of equal mass $m = 1.10^{-3}$. 
The total number of the discrete particles in the system is here $N=2000$.
We discretize the time integral of the external force $f$ by the left-hand rectangle method
\[ 
	\int_{t_h}^{t_h+\Delta t}{f(s,X_s) \,ds} 
		\approx \Delta t \,f(t_h,X^{\Delta t}_{t_h}),
\]
where $X^\Delta_{t_h} \in \R^N$ denotes the numerical solution 
at the discrete time $t_h := h\Delta t$ with $\Delta t>0$.	
Finally, we express the constraint $X \in \tK$ through the linear constraint
\[ 
	G \big( X^{\Delta t}_{t_h} - \tX \big) 
		\leq \begin{pmatrix} 0 \\ \vdots \\0 \end{pmatrix} 
\]
where the matrix $G \in \R^{N\times N}$ is given by
\[ 
	G = \begin{pmatrix}
		1  & -1      &  ~     & ~ & ~   \\
		~  & \ddots & \ddots & ~ & ~   \\
		~  & ~      & \ddots & \ddots & ~ \\
		~  & ~      & ~ & \ddots & -1 \\
		~  & ~      & ~ & ~ & 1
	\end{pmatrix}. 
\]
We represent in Figures~\ref{fig:congestion}
and~\ref{fig:decongestion2} our numerical solution during the four
phases of the process. We observe excellent agreement between the
numerical solution and the exact solution.

\section{Extension to heterogeneous maximal constraint}\label{sec:heterogeneity}

We consider the case where in the maximal density constraint, the
upper bound is replaced by a function, so that $0\leq \rho\leq\rho^*$,
where $\rho^*$ is transported with the flow. Thus
\begin{subnumcases}{\label{eq:granulaire_heterogeneous}}
	\partial_t \rho + \partial_x(\rho u) = 0 
\label{eq:granulaire_heterogeneous_mass}\\
	\partial_t(\rho u-\partial_x\gamma) + \partial_x \big( (\rho u-\partial_x\gamma)u \big) = \rho f 
\\
	\partial_t \rho^* + u\partial_x \rho^* = 0  
\label{eq:rho*}\\
	0 \leq \rho \leq \rho^* 
\\
	(\rho^* - \rho) \gamma = 0, \quad \gamma \leq 0.
\end{subnumcases}
This system has been studied by Degond et al.~\cite{degond2016} in the
Navier-Stokes framework. At initial time we prescribe, in addition to
$\rho_0, u_0$, and $\gamma_0$, the initial constraint $\rho^*_0$ so
that $\rho_0 \leq \rho^*_0$ a.e. in $\R$. Combining
\eqref{eq:granulaire_heterogeneous_mass} with~\eqref{eq:rho*}, we
observe that the ratio $r := \rho/\rho^*$ is conserved:
\begin{equation}
\partial_t r + \partial_x(r u) = 0.
\end{equation}
We now consider $r$ instead of $\rho$ and look for a monotone
transport map $Y_t$ such that
\[ 
	r_t =(Y_t)_\# r_0.
\]
We reformulate the system~\eqref{eq:granulaire_heterogeneous} in the form
\begin{subnumcases}{\label{eq:granulaire_heterogeneous_bis}}
	\partial_t r + \partial_x(r u) = 0 
\label{eq:granulaire_heterogeneous_mass_bis}\\
	\partial_t(ru-\partial_x\gamma) + \partial_x \big( (ru-\partial_x\gamma)u \big) = \rho f 
\\
	\partial_t \rho^* + u\, \partial_x \rho^* = 0
\\
	0 \leq r \leq 1
\\
	(1 - r) \gamma = 0, \quad \gamma \leq 0.
\end{subnumcases}
This transport $Y_t$ has thus to satisfy the constraint $Y_t \in \tK$
with $\tK : = K + \tY$. Here $K$ is again the cone of monotone
transport maps, and $\tY \in K$ is the uniquely determined transport
map with $\trho = \tY_\# r_0$, where $\trho$ is the same as in
\eqref{eq:cong_transport}. By replacing the density $\rho$
by the ratio $r = \rho/\rho^*$ in the proof presented in the previous
sections, we can define exactly in the same way as before a Lagrangian
velocity $U_t = \frac{d}{dt}Y_t$ and an adhesion potential such that
\[
	\bar r(y)\bar U(y) = \bar r(y) 
		\bigg( U_t(y) - \int_0^t f(s,Y_s(y)) \,ds \bigg) - \partial_y \Gamma_t(y).
\]
for a.e. $y\in\R$ and $t\in(0,T)$, with $\bar r, \bar U$ suitable
initial data. One can then check that we have constructed a global
weak solution to the heterogeneous
system~\eqref{eq:granulaire_heterogeneous_bis}.

\begin{theorem}	
Let $T > 0$ and external force $f \in L^\infty\big(0,T;\Lip(\R)\cap
L^\infty(\R)\big)$ be given. Suppose that $\bar\rho,\rho_0^* \in
\mathcal{P}_2(\R)$ with $\bar\rho,\rho_0^*\ll \mathcal{L}^1$ and
$\rho_0^*>0$ a.e., and assume that
\[
	0\leq \bar\rho\leq \rho_0^*
	\quad\text{a.e. in $\R$.}
\]
It follows that $r_0 := \rho_0^*/\bar\rho \leq 1$. Let $\bar u \in
L^2(\R,\bar\rho)$ and define
\[
	\rho_0 := \bar\rho, \quad	U_0 := \bar u, \quad Y_0 := \id.
\]
There exists a curve $[0,T] \ni t \mapsto Y_t \in \tK$ that is
differentiable for a.e. $t \in (0,T)$ and solves
\[
	Y_t = \mathrm{P}_{\tK}\left( Y_0 + \int_0^t{\Uf_s \,ds} \right),
	\quad
	\Uf_t = U_0 + \int_0^t f(s, Y_s) \,ds.
\]
The following quantities are well-defined: $ \rho^*_t(x) :=
\rho^*_0(Y_t^{-1}(x))$ for a.e. $x\in\R$, and
\[
	U_t(y) := \dot{Y}_t(y), \quad 
	\Gamma_t(y) := \int_{-\infty}^{y}{\Big(U_t(z) - \Uf_t(z)\Big) r_0(z) \,dz}
\]
for $y\in\R$ and a.e. $t \in (0,T)$. There exist $(u_t,\gamma_t)
\in\mathcal{L}^2(\R, r_t)\times W^{1,1}(\R)$, such that
\[ 
	U_t = u_t\circ X_t, \quad \Gamma_t = \gamma_t \circ X_t,
\]
where $r_t := (X_t)_\# r_0$. The tuple $(r, u, \gamma, \rho^*)$ is
a global weak solution of system~\eqref{eq:granulaire_heterogeneous_bis}.
\end{theorem}

\paragraph{Numerical simulation.} We consider the following initial data 
(see Figure~\ref{fig:var-rhomax1}):
\begin{gather*}
	\rho^*_0(x) = 1 + 0.2*\big(1-\cos(2\pi(x-0.5))\big), \\
	\rho_0(x) = 0.8 \,\rho^*_0(x) \,\mathbf{1}_{[0,1]}(x).
\end{gather*}
We add a compressive external force such that
\[ 
f(x) = \begin{cases}
       +0.5 \quad & \text{if $x < 0.5$} \\
       -0.5 \quad & \text{if $x > 0.5$,}
       \end{cases}
\]
which tends to concentrate the density at the middle of the interval $[0,1]$.
We set the particle mass $m = 1.10^{-3}$, 
so that $N=1000$ discrete particles are used in the simulation. 
We display in Figures~\ref{fig:var-rhomax1} and~\ref{fig:var-rhomax2} the concentration phenomenon 
with the appearance of a congested zone where the velocity is equal to $0$ 
and the adhesion potential is negative.
 
\begin{figure}[t]
	\begin{center}
		\includegraphics[width = 7.5cm]{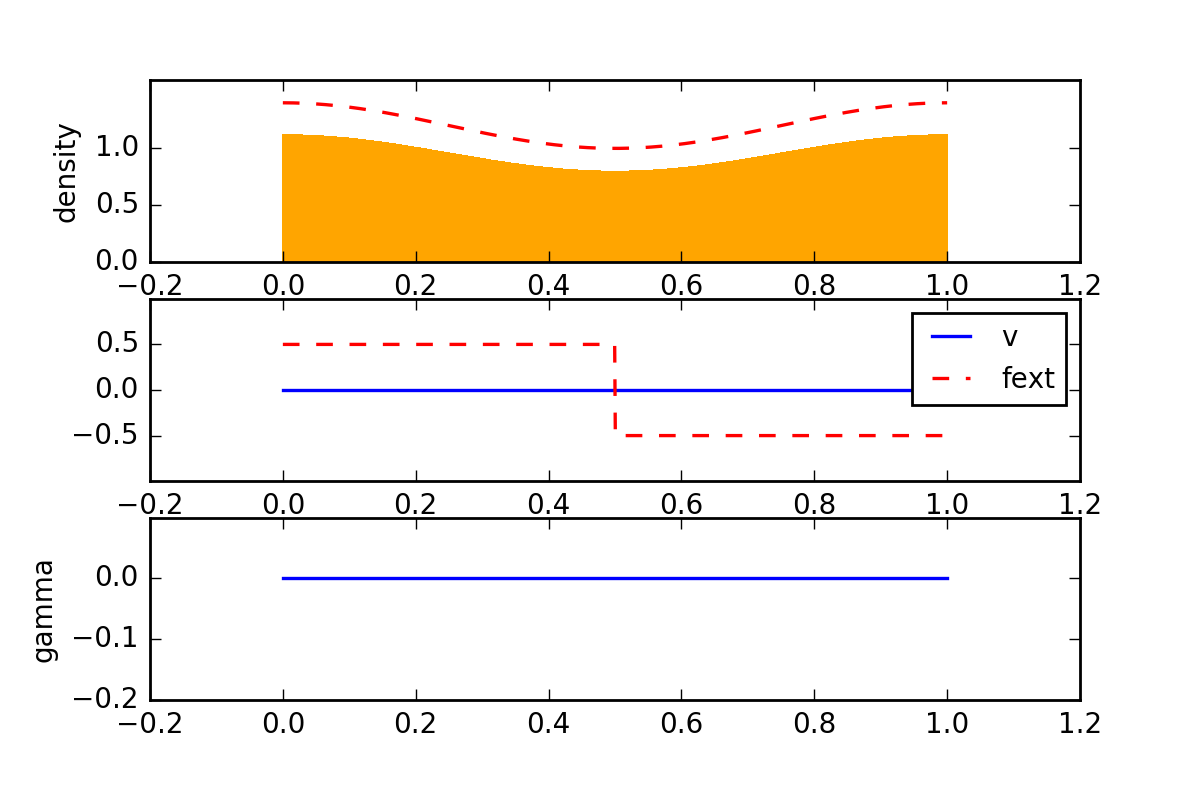} 
		\includegraphics[width = 7.5cm]{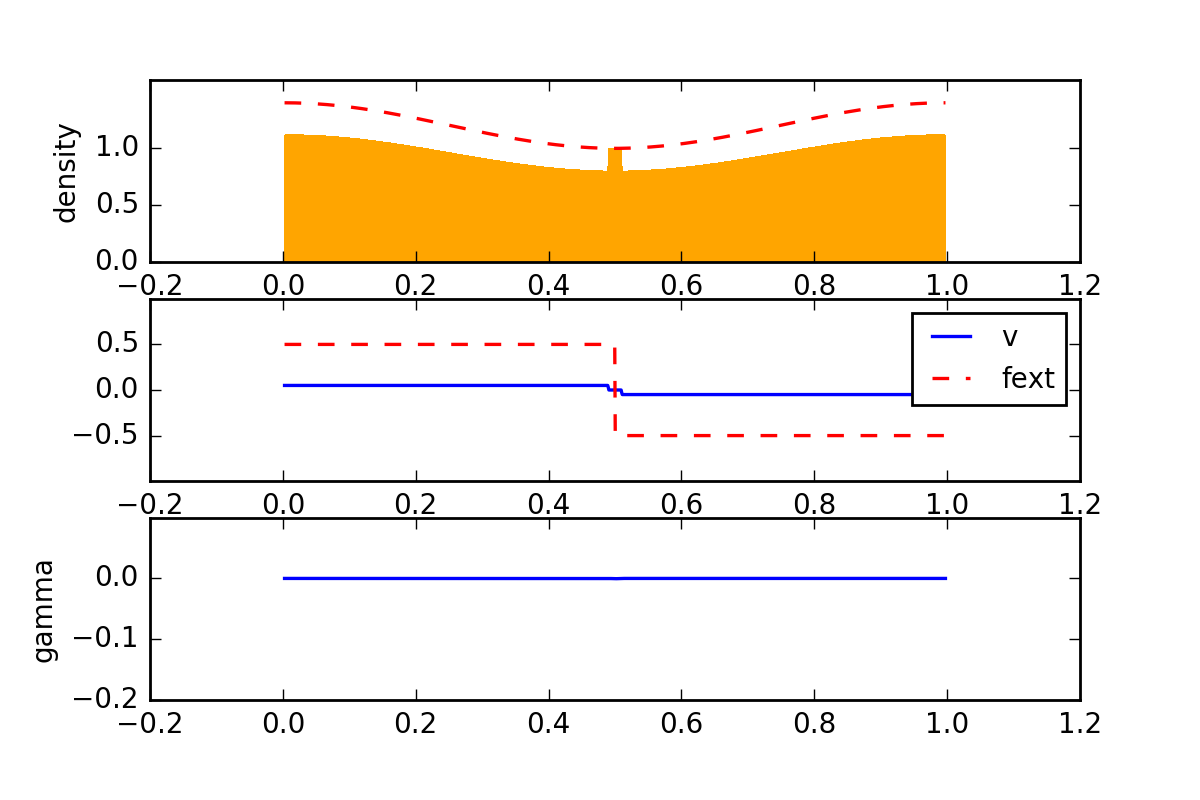}
	\end{center}
	\caption{{\label{fig:var-rhomax1}} Variable maximal density: solution at times $t = 0$ (left) 
	and $t=0.1$ (right).}
\end{figure}
\begin{figure}[t]
	\begin{center}
		\includegraphics[width = 7.5cm]{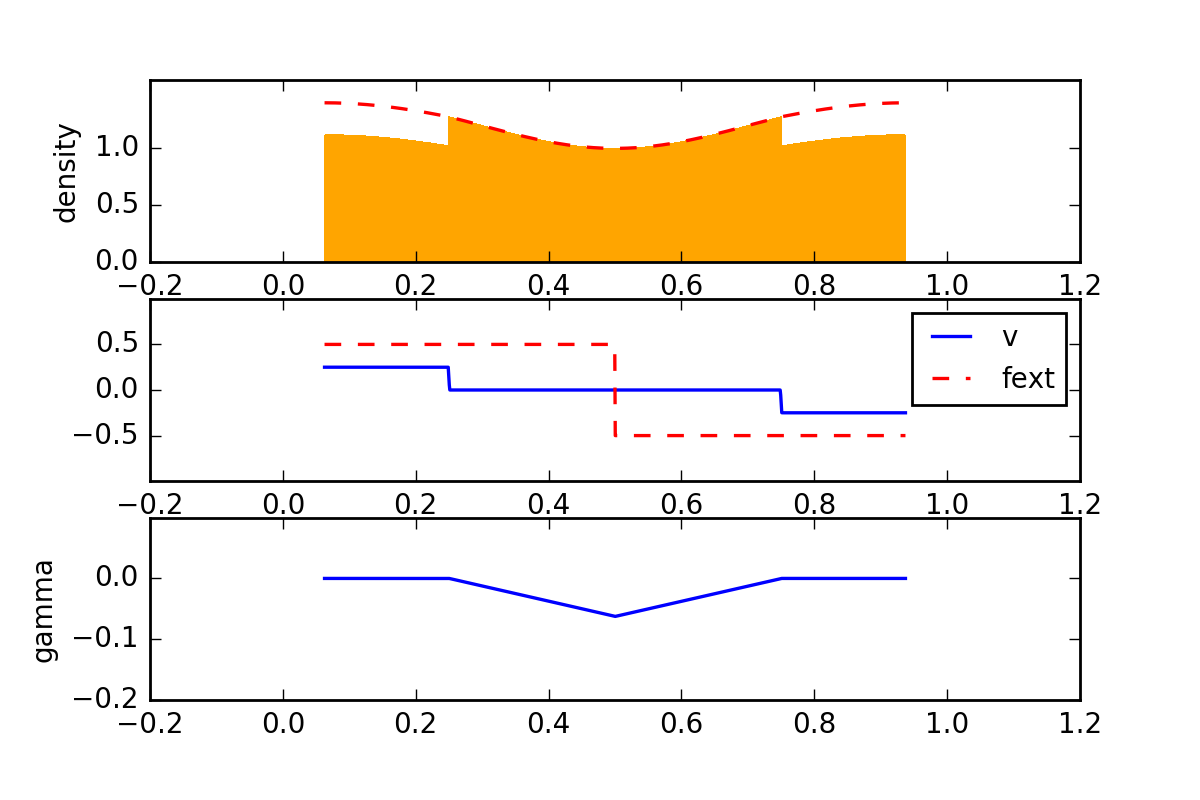} 
		\includegraphics[width = 7.5cm]{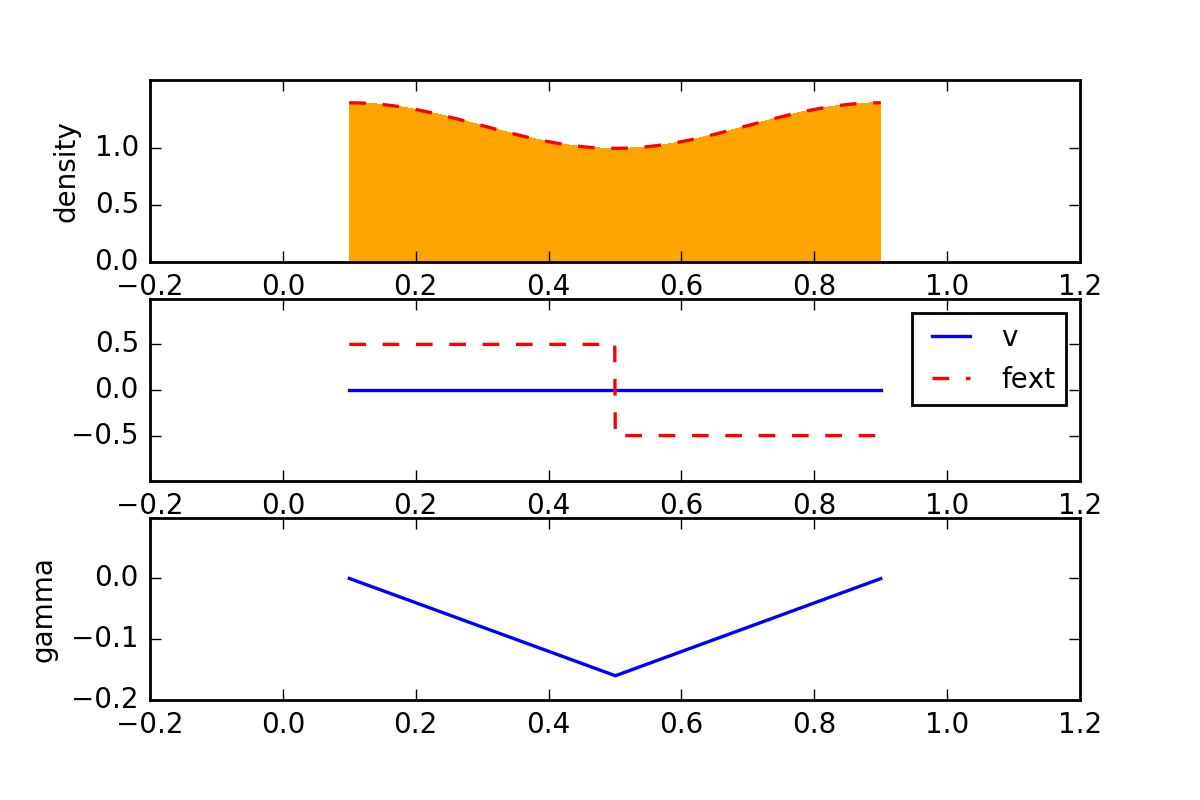}
	\end{center}
	\caption{{\label{fig:var-rhomax2}} Variable maximal density: solution at times $t = 0.5$ (left)
	and $t=0.8$ (right).}
\end{figure}



\begin{thebibliography}{10}
	
	\bibitem{ambrosio2000}
	{\sc Ambrosio, L., Fusco, N., and Pallara, D.}
	\newblock {\em Functions of bounded variation and free discontinuity
		problems}.
   \newblock The Clarendon Press Oxford University Press, 2000.

	\bibitem{ambrosio2008}
	{\sc Ambrosio, L., Gigli, N., and Savar{\'e}, G.}
	\newblock {\em Gradient flows: in metric spaces and in the space of probability
		measures}.
	\newblock Springer Science \& Business Media, 2008.
	
	\bibitem{bainov2013}
	{\sc Bainov, D.~D., and Simeonov, P.~S.}
	\newblock {\em Integral inequalities and applications}, vol.~57.
	\newblock Springer Science \& Business Media, 2013.
	
	\bibitem{berthelin2002}
	{\sc Berthelin, F.}
	\newblock Existence and weak stability for a pressureless model with unilateral
	constraint.
	\newblock {\em Mathematical Models and Methods in Applied Sciences 12}, 02
	(2002), 249--272.
	
	\bibitem{bouchut1994}
	{\sc Bouchut, F.}
	\newblock On zero pressure gas dynamics, advances in kinetic theory and
	computing, 171-190.
	\newblock {\em Ser. Adv. Math. Appl. Sci 22\/} (1994).
	
	\bibitem{bouchut2000}
	{\sc Bouchut, F., Brenier, Y., Cortes, J., and Ripoll, J.-F.}
	\newblock A hierarchy of models for two-phase flows.
	\newblock {\em Journal of NonLinear Science 10}, 6 (2000), 639--660.
	
	\bibitem{boudin2000}
	{\sc Boudin, L.}
	\newblock A solution with bounded expansion rate to the model of viscous
	pressureless gases.
	\newblock {\em SIAM Journal on Mathematical Analysis 32}, 1 (2000), 172--193.
	
	\bibitem{brenier2013}
	{\sc Brenier, Y., Gangbo, W., Savar{\'e}, G., and Westdickenberg, M.}
	\newblock Sticky particle dynamics with interactions.
	\newblock {\em Journal de Math{\'e}matiques Pures et Appliqu{\'e}es 99}, 5
	(2013), 577--617.
	
	\bibitem{brenier1998}
	{\sc Brenier, Y., and Grenier, E.}
	\newblock Sticky particles and scalar conservation laws.
	\newblock {\em SIAM journal on numerical analysis 35}, 6 (1998), 2317--2328.
	
	\bibitem{cavalletti2015}
	{\sc Cavalletti, F., Sedjro, M., and Westdickenberg, M.}
	\newblock A simple proof of global existence for the 1d pressureless gas
	dynamics equations.
	\newblock {\em SIAM Journal on Mathematical Analysis 47}, 1 (2015), 66--79.
	
	\bibitem{degond2016}
	{\sc Degond, P., Minakowski, P., and Zatorska, E.}
	\newblock Transport of congestion in the two-phase compressible/incompressible
	flow.
	\newblock {\em Nonlinear Analysis: Real World Applications, Elsevier} 42 (2018), 485--510.
	
	\bibitem{degond2017}
	{\sc Degond, P., Minakowski, P., Navoret, L. and Zatorska, E.}
	\newblock Finite Volume approximations of the Euler system with variable congestion.
	\newblock {\em Computers \& Fluids, Elsevier} 169 (2018), 23|-39.
	
	\bibitem{jabin2016}
	{\sc Jabin, P.-E., and Rey, T.}
	\newblock Hydrodynamic limit of granular gases to pressureless {E}uler in
	dimension 1.
	\newblock {\em  Quart. Appl. Math.} 75 (2017), no.1, 155--179.
	
	\bibitem{lannes2016}
	{\sc Lannes, D.}
	\newblock On the dynamics of floating structures.
	\newblock {\em  Annals of PDE}, Springer 3.1 (2017),11.
	
	\bibitem{lefebvre2007}
	{\sc Lefebvre, A.}
	\newblock {\em Mod{\'e}lisation num{\'e}rique d'{\'e}coulements
		fluide-particules: prise en compte des forces de lubrification}.
	\newblock PhD thesis, Universit{\'e} de Paris-Sud. Facult{\'e} des Sciences
	d'Orsay (Essonne), 2007.
	
	\bibitem{lefebvre2011}
	{\sc Lefebvre-Lepot, A., and Maury, B.}
	\newblock Micro-macro modelling of an array of spheres interacting through
	lubrication forces.
	\newblock {\em Advances in Mathematical Sciences and Applications 21}, 2
	(2011), 535.
	
	\bibitem{maury2007}
	{\sc Maury, B.}
	\newblock A gluey particle model.
	\newblock In {\em ESAIM: Proceedings\/} (2007), vol.~18, EDP Sciences,
	pp.~133--142.
	
	\bibitem{maury2015}
	{\sc Maury, B., and Preux, A.}
	\newblock Pressureless {E}uler equations with maximal density constraint : a
	time-splitting scheme.
	\newblock {\em https://hal.archives-ouvertes.fr/hal-01224008\/} (2015).
	
	\bibitem{natile2009}
	{\sc Natile, L., and Savar{\'e}, G.}
	\newblock A wasserstein approach to the one-dimensional sticky particle system.
	\newblock {\em SIAM Journal on Mathematical Analysis 41}, 4 (2009), 1340--1365.
	
	\bibitem{perrin20161d}
	{\sc Perrin, C.}
	\newblock Modelling of phase transitions in one--dimensional granular flows.
	\newblock {\em ESAIM: Proceedings and Surveys}, 58 (2017), 78--97.

	
	\bibitem{perrin2016AMRX}
	{\sc Perrin, C.}
	\newblock Pressure-dependent viscosity model for granular media obtained from
	compressible Navier–Stokes equations.
	\newblock {\em Applied Mathematics Research eXpress 2016}, 2 (2016), 289--333.
	
	\bibitem{perrin2015}
	{\sc Perrin, C., and Zatorska, E.}
	\newblock Free/congested two-phase model from weak solutions to
	multi-dimensional compressible Navier-Stokes equations.
	\newblock {\em Communications in Partial Differential Equations 40}, 8 (2015),
	1558--1589.
	
	\bibitem{preux2016}
	{\sc Preux, A.}
	\newblock {\em Transport optimal et équations des gaz sans pression avec
		contrainte de densité maximale}.
	\newblock Theses, Université Paris-Sud, Nov. 2016.
	
	\bibitem{rachev1998}
	{\sc Rachev, S.T., and Rüschendorf, L.}
	\newblock {\em Mass transportation problems}, vol 1
	Probability and its Applications,
	\newblock Springer-Verlag, New York, 1998. Theory.
	
	\bibitem{santambrogio2015}
	{\sc Santambrogio, F.}
	\newblock {\em Optimal transport for applied mathematicians}, vol.~87.
	\newblock Springer, 2015.
	
	\bibitem{sobolevskii1997}
	{\sc Sobolevskii, A.~N.}
	\newblock The small viscosity method for a one-dimensional system of equations
	of gas dynamic type without pressure.
	\newblock {\em Dokl. Akad. Nauk 356}, 3 (1997), 310--312.
	
	\bibitem{wolansky2007}
	{\sc Wolansky, G.}
	\newblock Dynamics of a system of sticking particles of finite size on the
	line.
	\newblock {\em Nonlinearity 20}, 9 (2007), 2175.
	
	\bibitem{zeldovich1970}
	{\sc Zel'Dovich, Y.~B.}
	\newblock Gravitational instability: An approximate theory for large density
	perturbations.
	\newblock {\em Astronomy and astrophysics 5\/} (1970), 84--89.
	
\end{thebibliography}
\end{document}